\documentclass[a4paper,10pt]{amsart}

\usepackage{amssymb}
\usepackage{amsmath}
\usepackage{amscd}
\usepackage{amsthm}
\usepackage{graphicx}
\usepackage[all]{xy}

\usepackage{color}          
\usepackage{epsfig}

\usepackage{lineno}

\newcommand{\C} {\mathbb C}

\newcommand{\cal}{\mathcal}

\newtheorem{theorem}{Theorem}[section]
\newtheorem{lemma}[theorem]{Lemma}
\newtheorem{corollary}[theorem]{Corollary}
\newtheorem{proposition}[theorem]{Proposition}

\theoremstyle{definition}

\title{On the fixed points of the Ruelle operator}
\author{Carlos Cabrera and Peter Makienko}

\begin{document}
\begin{abstract}
We discuss the relation between the existence of fixed points of the Ruelle 
operator acting on different Banach spaces, with Sullivan's 
conjecture in holomorphic dynamics.  
 \end{abstract}
\maketitle
\footnotetext{This work was partially supported by PAPIIT IN-102515 
and CONACYT CB2015/255633.}

\section{Introduction}

Let $Rat_d(\overline{\C})$ be the set of all rational maps on the 
Riemann sphere $\overline{\C}$ of  degree $d$. 
When $R$ is an element  of $Rat_d(\overline{\C})$, the postcritical set 
of $R$ is given by $$
P_R=\overline{\bigcup_k\bigcup_i R^i(c_k)},
$$  
where the union is taken over all critical points $c_k$ of $R$ and $i>0$. 
The \textit{Julia set} $J_R$ is the set of accumulation points of all 
periodic points of $R$, with all isolated points removed. The Fatou set is 
given by $F_R=\overline{\C}\setminus J_R$.
The map $R$ is called \textit{hyperbolic} when  $P_R\cap J_R=\emptyset$.  
The Fatou conjecture states: \textit{Hyperbolic maps are open and everywhere 
dense in $Rat_d(\overline{\C})$.}
 
Recall that a rational map $R$ is \textit{$J$-stable} if 
there is an open neighbourhood $U$ of $R$ in $Rat_d(\overline{\C})$, 
such that for $Q \in U$ there exists a homeomorphism 
$h_Q:J_R\rightarrow J_Q$, quasiconformal in Pesin's sense (also known as  metric quasiconformal), with 
$$Q=h_Q\circ R\circ h_Q^{-1}.$$

Due to a result of R. Ma\~n\'e, P. Sad and D. Sullivan (see \cite{MSS}),  the set of $J$-stable 
maps forms an open and everywhere dense subset of $Rat_d(\overline{\C})$, even 
more, a $J$-stable map is hyperbolic if and only if there is no invariant 
Beltrami differential supported on the Julia set.
Since hyperbolic maps are $J$-stable,  the Fatou conjecture 
becomes: \textit{Every $J$-stable map is hyperbolic.} 

An  \textit{invariant Beltrami differential}  $\mu$ is a 
$(-1,1)$-differential form locally expressed by $\mu(z)\frac{
d\overline{z}}{dz}$ and whose coefficient $\mu(z)$ is an $L_\infty$ function 
satisfying 
$$
\mu(z)=\mu(R(z))\frac{\overline{R'(z)}}{R'(z)}.
$$ 
That is, $\mu$ is a fixed point of the Beltrami operator, as defined below, 
acting on the space $L_\infty(\overline{\C})$ with respect to the planar 
Lebesgue measure. In this paper, whenever is clear from the context, 
we will denote both the differential and its coefficient by the same letter.

Sullivan's conjecture states: \textit{There exists an invariant Beltrami 
differential supported on the Julia set if and only if $R$ is a flexible 
Latt\`es map.} For the definitions and further properties of Latt\`es 
maps, see Milnor's paper \cite{MilLat}. 

Note that if $R$ is hyperbolic then $R^n$ is hyperbolic and hence $J$-stable 
for every $n>1.$ We have the following statement (see Theorem 2.1 in \cite{CMdecomp}): If there exists  $n>1$ such that the iterated map $R^n$ is 
$J$-stable then $R$ is hyperbolic. Therefore the Fatou conjecture is true when 
one considers an iterated rational map. On the other hand, Sullivan's 
conjecture predicts not only the absence of fixed points for the Beltrami 
operator but also the lack of periodic points for this operator.
In other words, there is no eigenvalue of the Beltrami operator which is a root 
of unity. Hence Sullivan's conjecture can be interpreted as a spectral 
problem for a semigroup of Beltrami operators.

According to Sullivan's dictionary between Kleinian groups and 
holomorphic dynamics, rational maps correspond to finitely 
generated Kleinian groups. In this setting, we can reformulate Teichm\"uller 
theorem (see \cite{GardLakic}) as follows:

\textit{Let $\Gamma$ be a finitely generated Kleinian group. For every 
$\Gamma$-invariant Beltrami differential $\mu$ inducing a non-trivial 
quasiconformal deformation (see definition below) on the associated Riemann 
surface $S_\Gamma$, 
there exists a $\Gamma$-invariant holomorphic $2$-form $\phi$ such that $\int_F 
\mu \phi\neq 0$ on any fundamental domain $F$ of $\Gamma$.}

In other words the following separation principle holds: The space of 
invariant holomorphic $2$-forms separates the space of quasiconformally 
non-trivial invariant Beltrami differentials. 

The separation principle is well-known in ergodic theory for bounded semigroups 
of linear endomorphisms of a Banach space. In fact, it is the subject of many 
ergodic theorems and is one of the oldest principles in this theory.

Due to the observations above, in this article we discuss the following 
question:

\emph{Given a representation of the dynamics of $R$ into a semigroup of contractions 
of a suitable Banach space, what consequences do arise from the existence of a 
common non-trivial fixed point of such a representation? }

Furthermore, we will discuss what happens when these representations satisfy a 
separation principle (definitions will be given below).

To keep in line with Sullivan's conjecture we consider representations that 
arise as versions of complex pull-back or push-forward 
operators acting on either invariant subspaces $X \subset L_p(W)$ (not 
necessarily closed) or on spaces which are predual, dual or bidual to $X$. 
Here $W$ is an  $R$-invariant set (that is $R(W)\subset W$)  of positive 
Lebesgue measure in the Riemann sphere and $1\leq p \leq \infty$.

Throughout our discussion, unless otherwise stated, 
all $L_p$ spaces are taken with respect to the planar Lebesgue measure on the 
dynamical plane.  Also we will use one or some combination of the following restrictions on $R$:
\begin{enumerate}
 \item The postcritical set $P_R\neq \overline{\C}.$
 \item The postcritical set has Lebesgue measure $0.$
 \item The intersection $P_R\cap F_R$ is finite and the Fatou 
 set $F_R$ does not contain parabolic or rotational domains. 
 
\end{enumerate}

In each section we will specify which restrictions apply. But let us note that the class
of rational maps satisfying the restrictions (1)-(3) is still relevant for the Fatou-Sullivan
conjecture.  In fact, when $F_R$ has no rotation nor parabolic domains and $J_R$ is connected, 
an application of quasiconformal 
surgery   gives that  the $J$-stability component of $R$ contains a map $Q$ which does 
not admit  non-trivial quasiconformal deformations on the Fatou set. 
Moreover, by Ma\~n\'e-Sad-Sullivan's theorem $Q$ is unique up to 
M\"obius conjugacy only in the case when $J_Q$ does not support an invariant Beltrami differential. 

Also, after the celebrated examples of Julia sets of positive measure 
given by X. Buff and A. Cheritat of Cremer polynomials, and by A. Avila and M. 
Lyubich of infinitely renormalizable polynomials, the main conjecture is that 
the postcritical set either has measure zero or is the whole Riemann sphere. So if
$F_R\neq \emptyset$, the restriction (2) is natural. 

We brief the main theorems of the article into the following two theorems. We say that 
an integrable function $f$  is \textit{regular} whenever its $\overline{\partial}$ derivative, 
in the sense of distributions, is a finite 
complex valued measure. Examples of non-regular functions are given by characteristic 
functions of suitable compact subsets of $\C$ (see the discussion after Theorem  
\ref{thm.regularfixed}  and Proposition \ref{prop.nonregular}). 

Other notations and definitions will be given in Section 2.
\begin{theorem}[Fixed Points of Ruelle Operator] 
Let $R$ be a rational map. Then the following hold true.
\begin{enumerate}
 \item No simple function is invariant under Ruelle operator $R^*.$ Moreover, if
 $S_R$ is connected. Then a non-zero regular function $f$ is a fixed point of $R^*$
 if and only if $R$ is a flexible Latt\`es map.
 \item (The $L_p$-case)  Let  $p$ and $q$ be such that  $1<p<\infty$ and  $\frac{1}{p}+\frac{1}{q}=1.$
 The operator $R^{*}_p:L_p(\C)\rightarrow L_p(\C)$ given by $$ 
R^{*}_{p}(\phi)=\frac{1}{\sqrt[q]{d}}\sum 
\phi(\zeta_i)\frac{\zeta_i'}{\overline{\zeta_i '}}|\zeta_i'|^{\frac{2}{p}},
$$ 
has a non-zero fixed point if and only if $R$ is a flexible Latt\`es map.
\item If $R$ is mixing with respect to any finite invariant measure absolutely continuous and $P_R\neq \overline{\C}$ then $R^*$ has a non-zero fixed point in $L_1(J_R)$ if and only if $R$ is a flexible Latt\`es map. 
\end{enumerate}
\end{theorem}

Part (1) follows from Theorem \ref{thm.regularfixed} and Corollary \ref{cor.simplefunc}. 
Note that Theorem \ref{thm.regularfixed} provides the more general case when $S_R$ is not 
connected.
Part (2) is the content of Theorem \ref{thm.lpcase}.
Part (3) is Theorem \ref{th.mixing}.

For every critical value  $v$  of $R$  define the operator 
$E_v:L_\infty(\C)\rightarrow \ell_\infty$ by the formula 
$$
E_v(\psi)=\left( \int_{\C} \psi(z)A_{n}(\gamma_v)(z)|dz|^2\right )_{n=0}^{\infty},
$$ where $A_n(\gamma_v)=\frac{1}{n}\sum_{i=0}^{n-1} R^{*i}(\gamma_v)$ 
is the C\'esaro averages and $\gamma_v(z)=\frac{v(v-1)}{z(z-1)(z-v)}$.

\begin{theorem}(Invariant Beltrami differentials). Given a rational map $R$, the following 
hold true.
\begin{enumerate}
 \item If $P_R$ has zero Lebesgue measure then $R$ satisfies Sullivan's conjecture if and 
 only if $R^*$ is mean-ergodic on $Hol(R)$ with the topology inherited from $L_1(J_R).$
 \item Assume $P_R\neq \overline{\C}$ and that there are no rotational domains. If $T$ is the 
 Thurston operator for $R$ then $T:B_0(S_R)\rightarrow B_0(S_R)$ is mean-ergodic. 
 Moreover, let $\alpha\in B(S_R)$ with $T(\alpha)=\alpha$ and $\|\alpha\|_T=1$, 
 where $\|\cdot \|_T$ is the Teichmuller norm. If 
 $$\inf_{\phi\in B_0(S_R)} \|\alpha -\phi\|_T<1$$ then
$R$ is a flexible Latt\`es map.
\item Assume $P_R\neq \C$, the set $P_R\cap F_R$ is finite  and $P_R$ does not admit 
finite invariant absolutely continuous complex valued measures. Then $R$ satisfies 
Sullivan's conjecture if and only if the operator $E_v$ is compact for every critical value $v.$ Moreover, 
the operator $(Id-T):B(S_R)\rightarrow B(S_R)$ is compact if and only if $R$ is postcritically finite. 
\end{enumerate}
\end{theorem}

Part (1) follows from Theorem \ref{th.MeanErg.1} which is given in a more general situation where $P_R$ is allowed to have positive Lebesgue measure. 
Part (2) follows from Theorem \ref{th.ThurstonMeanergodic} and Corollary \ref{cor.Teichflexible}.
Part (3) follows from Theorem \ref{th.Evcompact} and Corollary \ref{cor.Evcompact}.

The article also includes several results not mentioned in the theorems above which 
have independent interest, for example Proposition \ref{prop.contId} compares the 
masses over the Julia and the Fatou sets for $\gamma\in Hol(R).$ Also see Theorem
\ref{th.dissipative}, Theorem \ref{th.kukareko}, Theorem \ref{th.coseparable} and
the discussions in the respective sections. 

Most results in the paper are given in terms of ergodic theory and suggest that the 
Sullivan's conjecture holds true.  

\par\medskip\noindent \textbf{Acknowledgments.}
The authors would like to thank A. Poirier and  anonymous referees  for useful 
comments and suggestions on an early draft of this paper.

\section{The Pull-back and Push-forward actions.}

In this section we give the definition of the Ruelle transfer operator for 
a rational map $R$ which is a complex version of the Perron-Frobenius operator. 
Perhaps, the first instance of this operator appeared in 
Ruelle's paper \cite{RuelleZeta}. 

We take the standpoint of the theory of quasiconformal deformations, which deals 
with a dual interplay between Beltrami differentials and quadratic 
differentials (i.e, between some $(-1,1)$-forms and $(2,0)$-forms). 
In fact, there is a natural action on both spaces  induced by a degree 
$d$ rational function $R$, once we think of $R$ as a ``local change of 
variables''. 

Let $F_{m,n}$ be the space of all $(m,n)$ forms $\alpha(z)=\phi(z)dz^md\overline{z}^n $ 
where $\phi$ is a complex valued measurable function on $\C$. The 
\textit{pull-back} operator acting 
on $F_{m,n}$ 
is given by $$R_{*(m,n)}(\alpha)=\alpha\circ R= 
\phi(R(z)) R'(z)^m \overline{R'(z)}^ndz^m
d\overline{z}^n.$$ The \textit{push-forward} operator on $F_{m,n}$ is given 
by $$R^*_{(m,n)}(\alpha)=\sum \alpha(\zeta_i)$$ where the sum is taken over all 
branches $\zeta_i$ of  $R^{-1}.$ Therefore we have  $R^*_{m,n}\circ R_{*\,m,n}= 
deg(R) Id.$ 

The \textit{Beltrami operator }is $B_R=R_{*(-1,1)}$. The 
operator  $|B_R|=R_{*(0,0)}$ is called the \textit{modulus} of  $B_R$ and 
satisfies 
$|B_R||\phi|=|B_R(\phi)|$ almost everywhere for every measurable function 
$\phi.$ 

The \textit{Ruelle transfer operator}, the \textit{Ruelle operator} for short, 
is $R^*=R^*_{(2,0)}$, while the operator $|R^*|=R^*_{(1,1)}$ is called the 
\textit{modulus of Ruelle operator} and satisfies $|R^*(\phi)|\leq |R^*||\phi|$ 
almost everywhere for every measurable function $\phi.$  The modulus of the Ruelle 
operator is also known as the \textit{Perron-Frobenius} operator for the map 
$R$ or as the \textit{push-forward} operator on the space of absolutely 
continuous measures.

Using coefficients the Beltrami operator  and its modulus are defined by 
the formulas
$$B_R(\phi)=\phi(R)\frac{\overline{R'}}{R'} \textnormal{ 
and }|B_R|(\phi)=\phi(R)$$ respectively, where $\phi$ is a measurable function 
with respect to the Lebesgue measure. 

In turn, the Ruelle operator and its modulus are given by

$$R^*(\phi)=\sum \phi(\zeta_i)(\zeta'_i)^2 \textnormal{ and } |R^*|(\phi)=\sum 
\phi(\zeta_i)|\zeta'_i|^2,$$ respectively.  Both  sums are taken over all 
branches $\zeta_i$ of $R^{-1}$.
We end this section with the following simple facts. 

\begin{proposition}\label{prop.bounds}  Let $A$ be a Lebesgue measurable set,  
$\mu\in L_\infty(\C)$ and $\phi\in L_1(\C)$. Then the following statements hold: 
\begin{enumerate}
 \item 
$$\int_A \mu(z) R^*(\phi(z))|dz|^2= \int_{R^{-1}(A)} 
B_R(\mu(z)) \phi(z)|dz|^2.$$
\item $$\int_A \mu(z) |R^*|(\phi(z))|dz|^2=\int_{R^{-1}(A)} 
\mu(R(z)) \phi(z)|dz|^2.$$

\item Define $R_*(\phi)=\frac{R_{*\,(2,0)}(\phi)}{deg(R)}=\frac{\phi(R) 
R'^{2}}{deg(R)}$, then
$$\int_A |R_*(\phi(z))| |dz|^2\leq \int _{R(A)} |\phi(z)| |dz|^2.$$

\item If $\phi$ is a holomorphic function outside the postcritical set 
$P_R$, 
then $R^*\phi$ is also a holomorphic function outside $P_R.$
\end{enumerate}
\end{proposition}
\begin{proof} Part (1). Fix a system of branches of $R^{-1}$ in the 
following way: let $\tau$ be any differentiable arc containing all critical 
values of $R.$ Take $D=\overline{\C} \setminus \tau$ then by the monodromy theorem 
each of the branches $\zeta_i$ of $R^{-1}$ defines a holomorphic function on $D$. 
Set $D_i=\zeta_i(D)$, then $D_i\cap D_j=\emptyset$ for $i\neq j$, so we have
\begin{align*}\int_A\mu(z)R^*(\phi(z))|dz|^2
&=\int_{A\cap D}\mu(z)R^*(\phi(z))|dz|^2 \\
&=\int_{A\cap D}\sum_i \mu(z)\phi(\zeta_i(z))(\zeta'_i(z)^2) |dz|^2,
\end{align*} 
after a change of variables the latter is equal to 
$$\sum_i \int_{\zeta_i(A\cap 
D)}\mu(R(z))\frac{\overline{R}'(z)}{R'(z)}\phi(z)|dz|^2=\int_{R^{-1}(A)}B_R(\mu(z))\phi(z)|dz|^2.$$
Similar computations prove parts (2) and (3).
 
Part (4). Again by the monodromy theorem, the Ruelle operator 
does not depends on the local choice of branches of $R^{-1}.$  Outside the postcritical set, 
every branch of $R^{-1}$ is a local holomorphic function and 
$R^{-1}(\overline{\C}\setminus P_R)\subset \overline{\C}\setminus P_R$, therefore $R^*(\phi)$ 
is sum of holomorphic functions, and so is holomorphic. \end{proof}

Let $A$ be an $R$-invariant set, that is $R(A)\subset A$, then 
$R^*$ acts on $L_1(A)$ with $\|R^*\|\leq 1$. Indeed, by extending every element
in $L_1(A)$ by $0$ on $\C\setminus A$, the space $L_1(A)$ can be regarded as a closed 
subspace of $L_1(\C)$. By definition, the support  $supp(R^* f)$ is contained in 
$R(supp(f)).$ If $A$ is completely $R$-invariant (i.e. $R^{-1}(A)=A$), then $B_{R}$ acts
on $L_\infty(A)$ with $\|B_R\|=1$ and  $B_R$ is dual to $R^*.$ It is known that 
if $f$ is a quasiconformal automorphism of $\overline{\C}$ then $R_f=f\circ R \circ f^{-1}$ 
is rational if and only if its Beltrami coefficient $\mu_f(z)=\frac{\overline{\partial}f}{\partial f}(z)$ 
satisfies $\mu_f(R(z))\frac{\overline{R'(z)}}{R'(z)}=\mu_f(z)$ almost everywhere. 
Hence by the measurable Riemann mapping theorem the unit ball of the space 
of fixed points of  $B_R$ in $L_\infty(\C)$ determines all quasiconformal deformations of the 
rational map $R.$

When  $A$ is backward $R$-invariant  (i.e. $R^{-1}(A)\subset A$) and $\phi$ 
is an integrable function on $A$, then $$\int_A |R^*(\phi)| \leq \int_A 
||R^*|(\phi)|\leq \int_A |R^*||\phi|= \int_{R^{-1}(A)}|\phi|\leq \int_A|\phi|.$$
Here, the integration is with respect to the Lebesgue measure. Therefore the Ruelle 
operator is a contraction in the $L_1$ norm.

As another entry of Sullivan's dictionary the reader may recognise the Ruelle 
operator as the relative Poincar\'e Theta operator for branched coverings of 
the sphere onto itself. 

\section{The Thurston operator}\label{sec.Thurs}

Following ideas from Teichm\"uller theory we consider the action of the Ruelle operator on spaces of functions that are holomorphic on a given open set.  

Let $K$ be a closed subset of $\overline{\C}$ and   let  
$H(K)$ be the subspace of $L_1(\overline{\C})$ of all functions holomorphic 
outside of $K$ with the restricted norm of $L_1(\overline{\C})$. If $K$ is  $R$-invariant and 
contains the postcritical set $P_R$ then the Ruelle 
operator $R^*$ is a contractive endomorphism of $H(K).$   Define
$S_K=\C\setminus K$, and let $A(S_K)$ be the space of all 
integrable holomorphic functions on $S_K$ equipped with the $L_1$ norm, so $A(S_K)$ is 
a Banach space.  By the discussion on the previous section the 
Ruelle operator also acts as a contracting endomorphism of $A(S_K)$.
Every element $f$ in $A(S_K)$ extends to an element in $H(K)$, just set 
$f(z)=0$ for all $z$ in $K$. This extension gives a canonical inclusion from 
$A(S_K)$ into $H(K)$, which is an isomorphism precisely when the Lebesgue 
measure of $K$ is $0.$ 

Let $B(S_K)$ denote the associated Bergman space, that is, the space of all 
holomorphic functions $\phi$ on $S_K$ with the following $L_\infty$-norm 
$$
\|\phi\|=\sup_{z\in 
S_K}|\lambda_K^{-2}(z)\phi(z)|
$$  where $\lambda_K$ denotes the complete hyperbolic metric on $S_K.$ By the 
Bers' isomorphism theorem (see Theorem 2.1 in Chapter 3, page 89 of \cite{Krabook}), the space 
$B(S_K)$ is linearly isomorphic to $A^*(S_K)$, the dual of the Banach space $A(S_K)$, by 
the correspondence that associates to every $\phi\in B(S_K)$ the continuous functional 
$l_\phi(\psi)=\int_{S_K}\lambda^{-2}\overline{\phi}\psi|dz|^2$ in $A^*(S_K)$. 
Furthermore, there is an equivalent norm on $B(S_K)$, called the \textit{Teich\-m\"uller 
norm}, which is the canonical supremum norm of continuous linear 
functionals on the unit sphere in $A(S_K)$.

Let $T:B(S_K)\rightarrow B(S_K)$ be the dual operator of $R^*$ up to the 
identification above. Then  $T$ is a power-bounded operator which is a 
contraction in the Teichm\"uller norm. We call the operator $T$ the 
\textit{infinitesimal Thurston pull-back operator}, or \textit{Thurston 
operator} for short.  Indeed, as it was shown by A. Douady and J. H. Hubbard in \cite{DHTop}, 
when $R$ is a postcritically finite rational map the operator 
$T$ is  the derivative of the Thurston pull-back map.

Let $B_0(S_K)$ be the subspace  of all elements in $B(S_K)$ vanishing at 
infinity. In other words $B_0(S_K)$ is the space of all $\phi$ in $B(S_K)$ 
such that $|\lambda_K^{-2} \phi(z_j)|$ converges to zero, whenever $z_j$ is a 
sequence converging to the boundary $\partial S_K.$ 

Let $A_*(S_K)$ be the subspace of $A^*(S_K)$ such that the dual space 
$(A_*(S_K))^*$ is isometrically isomorphic to $A(S_K)$ (see for instance 
Theorem 5 on page 52 of \cite{GardLakic}). The space $A_*(S_K)$ is constructed 
as follows. A sequence $\{\phi_j\}$ in $A(S_K)$ is \textit{degenerating} if 
$\|\phi_j\|=1$ for all $j$ and $\phi_j$ converges pointwise to $0$ on $S_K$. 
Then $A_*(S_K)$ is the kernel of the seminorm on $A^*(S_K)$ given by
$$\beta(l)=\sup (\limsup_i 
|l(\phi_i)|),$$
where the supremum is taken over all degenerating sequences $\{\phi_i\}$ in 
$A(S_K).$ 

Moreover, the $*$-weak 
topology on $A(S_K)$ induced by $A_*(S_K)$ coincides with the topology of 
pointwise convergence of bounded sequences.

The Bers' isomorphism theorem together with Theorem 1 of \cite{BonetWolf}, 
provides a correspondence between the topologies and the Banach structure of the spaces we 
are dealing with as follows:

\begin{equation*}
\begin{split}
B(S_K)\simeq& B_0(S_K)^{**} \overset{f^{**}}\longleftrightarrow 
A^*(S_K)\overset{g^{**}}{\longrightarrow} \ell_\infty=c_0^{**},\\
&B_0(S_K)^{*}\overset{f^*}{\longleftrightarrow} 
A(S_K)\overset{g^*}{\longleftarrow}\ell_1,\\
&B_0(S_K)\overset{f}{\longleftrightarrow} 
A_{*}(S_K)\overset{g}{\longrightarrow} 
c_0.
\end{split}
\end{equation*}

In the notation above,  $f$ is the restriction map of the Bers' isomorphism to $A_*(S_K)$ 
and, by Lemma 1 and Corollary 1 of pages  259-260  in \cite{GardLakic}, it
is a surjective map onto $B_0(S_K)$. There is an isomorphism $h$ to its image 
from $B_0(S_K)$ into $c_0$ given by Theorem 1 in \cite{BonetWolf}. The 
map $g$ is just the composition $h\circ f$. Here $\ell_\infty$, $\ell_1$ and 
$c_0$ denote the spaces of complex valued sequences that are bounded, 
absolutely summable and converging to $0$, respectively. In this work, we 
often identify the space $A(S_K)$ with $B^*_0(S_K)$ using  Bers' isomorphism, 
so the action $R^*:B^*_0(S_K)\rightarrow B^*_0(S_K)$ is well defined.

Now we collect some facts about the geometry and dynamics of operators on  
$B(S_K)$ and $B_0(S_K)$. First we need some definitions. 
A Banach space $B$ is a \textit{Grothendieck space} if every $*$-weak 
convergent sequence of continuous functionals $\{l_i\}$ also converges 
in the weak topology on the Banach space $B^{*}$ dual to $B$.

Every complemented closed subspace of a Grothendieck space is again a 
Grothen\-dieck space. Clearly, every reflexive space is a Grothen\-dieck space. 

A Banach space $B$ has the \textit{Dunford-Pettis property} if every weakly 
compact operator from $B$ into any Banach space maps weakly compact sets into 
norm compact sets. As in the case of Grothendieck spaces the Dunford-Pettis 
property is also inherited on complemented closed subspaces. 
A typical example of a Grothendieck space with the Dunford-Pettis property is 
$L_\infty(X,\mu)$ where $(X,\mu)$ is a positive measure space (see \cite{Lotz}).

A series $\sum x_n$ in a Banach space $X$ is called \textit{weakly unconditionally 
Cauchy (wuC)} if for every $l\in X^*$ the series $\sum|l(x_n)|$ is bounded. A 
Banach space $X$ is said to have the \textit{property (V) of Pe\l{}czy\'{n}ski} 
if 
every subset $K\subset X^*$ is relatively weakly compact whenever $K$ 
satisfies  $$\lim_n \sup_{l\in K}|l(x_n)|=0$$ for
every wuC series $\sum x_n$ in $X$. 

By results of functional analysis (see, for example, Corollary 3.7, page 132 in 
\cite{HarmandWWerner}) any closed subspace $Y\subset c_0$ has the property (V) 
of Pe\l{}czy\'{n}ski.

Hence we have the following 

\begin{itemize}
 \item $B_0(S_K)$ has the property (V) of  Pe\l{}czy\'{n}sky.
 \item $B(S_K)$ is a Grothendieck space with the Dunford-Pettis property.
\end{itemize}

Indeed, by results of J. Bonet and E. Wolf  in \cite{BonetWolf}, the space $B_0(S_K)$
is isomorphic to a closed subspace of $c_0.$ 

By the Bers' isomorphism theorem,  we 
have $L_\infty(S_K)\cong N\oplus B(S_K).$ So $B(S_K)$ is a complemented
subspace of a Grothendieck space with the Dunford-Pettis property. The space 
$N=A(S_K)^\perp\subset L_\infty(S_K)$ is the annhilator of $A(S_K)$. In 
other words, $N$ consists of the trivial Beltrami differentials.  

Again, by a combination of classical results in functional analysis we have the 
following fact:
\begin{itemize}
\item If $E:B_0(S_K)\rightarrow B_0(S_K)$ is a linear operator then either $E$ 
is compact or there exists an infinite dimensional subspace $Y$, which is 
isomorphic to $c_0$ such that $E|_Y$ is an isomorphism onto its image. 
\end{itemize}
Indeed, as noted above $B_0(S_K)$ has the property (V) of Pe\l{}czy\'{n}ski, 
then by the Lemma 3.3.A and the Theorem 3.3.B on page 128 of 
\cite{HarmandWWerner}, either $E$ is weakly compact or there exists an infinite 
dimensional subspace $Y$, which is isomorphic to $c_0$ such that $E|_Y$ is an 
isomorphism onto its image. However,  if $E$ is weakly compact then 
$E^*:B^*_0(S_K)\rightarrow B_0^*(S_K)$ is also weakly compact. Let us show 
that every bounded weakly convergent sequence $\{\psi_n\} \subset A(S_K)$ 
contains a norm convergent subsequence. In fact, every bounded sequence in 
$A(S_K)$ forms a normal family. Let $\{\psi_{n_k}\}$ be a subsequence of 
$\{\psi_n\}$ converging to its pointwise limit $\psi.$ Hence, $\psi_{n_k}$ is 
weakly convergent and by the Fatou lemma $\psi\in A(S_K).$ Then by the uniform 
integrable convergence in measure theorem (see Theorem 1.5.13 in 
\cite{Taomeasure}), $\psi_{n_k}$ converges to $\psi$ in norm. Therefore 
$E^*$ and, hence also $E$, are compact operators.

We say that an invariant Beltrami differential $\mu\in L_\infty(\bar{\C})$ defines a 
\textit{non-trivial quasiconformal deformation} if and only if 
$l_\mu(\psi)=\int_{S_K} \psi(z) \mu(z) |dz|^2$ is a non-zero 
 functional on $A(S_K)$, note that $l_\mu$ is $R^*$-invariant
 that is $l_\mu(R^*(\psi))=l_\mu(\psi)$ for all $\psi\in A(S_K)$. In other words, $l_\mu$ 
 induces a non-zero fixed point for $T$ on $B(S_K).$ 

Finally, in order to use results given by the second author in
\cite{MakRuelle} without cumbersome recalculations, throughout the paper we will 
assume that $\{0,1,\infty\}$ are fixed points of $R$.  This is always the case 
after passing to a suitable iteration of $R$ and conjugating with a M\"obius 
map. However most of our results do not need this normalization.
We fix $K_R= P_R\cup \{0,1\}$ and set $S_R=\C\setminus K_R$ and  always assume
that $S_R$ is  non-empty.

\section{Mean ergodicity in holomorphic dynamics}\label{section4}

Given an operator $S$ on a Banach space $X$,  the 
$n$-\textit{Ces\`aro averages} of $S$ are the operators $A_n(S)$ defined for 
$x\in 
X$ 
by $$A_n(S)(x)=\frac{1}{n}\sum_{i=0}^{n-1}S^i(x).$$

An operator $S$  on a Banach space $X$ is called  \textit{mean-ergodic} if $S$ 
is power-bounded, that is, it satisfies $\|S^n\|\leq M$ for some number $M$ 
independent of $n$, and the Ces\`aro averages $A_n(S)(x)$ converges in  norm for 
every $x\in X$.

The topology of convergence in norm is also called the strong topology on $X.$
If $A_n(S)$ converges uniformly on the closed unit ball on $X$ then the 
operator $S$ is called \textit{uniformly ergodic}. The following facts can be found, for example, 
in Krengel's book \cite{Krengel}.

\begin{enumerate}
 \item (\textbf{Separation principle)}. The operator $S$ is mean-ergodic if and only if $S$
satisfies the principle of separation of points: If $x^*$ is a fixed point of 
$S^*$ then there exists  $y\in X$ a fixed point of $S$ such that 
$\langle y,x^*\rangle \neq 0.$

\item  For a power bounded operator $S$, the equality $\lim_{n\rightarrow \infty} A_n(S)(x)=0$ holds if and only if 
$x\in  \overline{(Id-S)(X)}.$

\item (\textbf{Mean ergodicity lemma}). For a power bounded operator $S$, consider the convex hull $Conv(S,x)$ of 
the orbit of a point $x$ under $S$. Then $y$ is a weak accumulation point of  
$Conv(S,x)$ if and only if $y$ is a fixed point of $S.$ In this situation 
$A_n(S)(x)$ converges to $y$ in norm.

If $X$ is a dual space, then $y$ is a $*$-weak accumulation point of $Con(S,x)$ if and only if $y$ is a fixed point of $S.$ 
Here a dual space is a space isometrically isomorphic to $B^*$,
for some Banach space $B$ where the notion of $*$-weak topology is well 
defined. 

\item (\textbf{Uniform Ergodicity lemma}). A power-bounded operator $S$ acting 
on a Banach space $B$ is uniformly ergodic if and only if the subspace 
$(Id-S)B$ is  closed or, equivalently, if and only if the point $1$ 
either belongs to the resolvent set or it is an isolated eigenvalue of $S$. 
If $dim(Fix(S))$ is finite then the intersection of the spectrum of $S$ with 
the unit circle consists of finitely many isolated eigenvalues. 
\end{enumerate}

Let $Hol(R)$ be the space of all integrable, with 
respect to the planar Lebesgue measure, rational functions having poles in the 
forward orbit of the set  $V(R)\cup \{0,1\}$, where $V(R)$ is the set of all 
critical values of $R$. Equivalently, $Hol(R)$ consists of all rational functions with 
simple poles in the forward orbit of $V(R)\cup \{0,1\}$ and a zero at infinity 
of multiplicity at least three. Note that $Hol(R)$ is a normed vector space 
with the norm  inherited from  $L_1(\overline{\C})$. The  space  $Hol(R)$ is not 
complete and, by the Bers' approximation theorem (see Theorem 9 
of  \cite{GardLakic}), its completion is $H(K_R)$ and contains a canonical 
inclusion of $A(S_R)$. Moreover, the completion of $Hol(R)$ is equal to $A(S_R)$ 
if and only if the Lebesgue measure of $P_R$ is zero.

We recall some facts from ergodic theory which will be used in this 
work. A positive measure set $M\subset \overline{\C}$ is called \textit{wandering} when 
the sets  $\{R^{-k}(M)\}_k$ are pairwise almost disjoint Lebesgue 
measurable sets, that is $R^{-i}(M)\cap R^{-j}(M)$ has 
Lebesgue measure zero whenever $i\neq j$. Let $D(R)$ be the union of all 
wandering sets. The set $D(R)$ is called the \textit{dissipative 
set} and the complement $C(R)=\C\setminus D(R)$ is  the \textit{conservative set}. 
Similarly, a positive measure set $M\subset \overline{\C}$ is called \textit{weakly 
wandering} when there is a sequence $0=k_0<k_1<k_2<...$ such that the sets  
$R^{-k_i}(M)$ are pairwise almost disjoint Lebesgue measurable sets. The \textit{weakly dissipative} set $W(R)$ is 
the union of all weakly wandering sets and $SC(R)=\C\setminus W(R)$ is 
called the \textit{strongly conservative set}. 

The following proposition is a consequence of Theorem 4.6 and 
Theorem 4.11,  pages 141 and 144  in \cite{Krengel}, respectively.
 
\begin{proposition}\label{prop.Krengel} Let $R$ be a rational map. Then
the Lebesgue measure of the symmetric difference $R(SC(R))\triangle SC(R))$ is zero.  
Furthermore, if $SC(R)$ has positive measure there exists a non negative
integrable function $P$ which is positive on $SC(R)$ and  such that $P(z)|dz|^2$ 
defines an invariant probability measure supported on $SC(R).$ 
Moreover, if for a given non negative measurable function $\phi$ we have that 
$\phi(z) |dz|^2$ defines an invariant probability measure, then 
$supp(\phi)$ is contained in $SC(R).$
\end{proposition}
\begin{proof}
 Since the weakly dissipative set is backward $R$-invariant then  $R(SC(R))=SC(R)$ 
 almost everywhere. Hence the restriction $R:SC(R)\rightarrow SC(R)$ is a 
 null-preserving transformation. If $SC(R)$ has positive measure, then by Theorem 4.11, 
 page 144 in \cite{Krengel}, there exists a finite $R$-invariant measure $\nu$ on 
 $SC(R)$ which is equivalent to the Lebesgue measure. Let $P$ be the Radon-Nikodym 
 derivative of $\nu$ with respect to the Lebesgue measure, so $P(z)>0$ almost everywhere on $SC(R)$.  
 Also every finite invariant measure absolutely continuous with respect to Lebesgue is 
 $0$ on the weakly dissipative set. By Theorem 4.6,  page 141 in \cite{Krengel}, 
 the decomposition $\C=SC(R)\cup W(R)$ is unique up to measure. If $\phi$ is a 
 non-negative function so that $\phi |dz|^2$ defines a finite invariant measure then 
 $supp(\phi)$ is invariant and again by Theorem 4.11 of 
 \cite{Krengel} we have $supp(\phi)\subset SC(R).$
\end{proof}
Due to the classification of Fatou  components and Proposition 
\ref{prop.Krengel},  it follows  that both the conservative set $C(R)$
and the strongly conservative set $SC(R)$ intersect the Fatou set $F_R$ 
precisely at the union of all rotation domains cycles. Indeed, if a periodic 
Fatou component $D$ is not a rotation domain then $D$ consists of a union of 
wandering sets. 

Now assume $D$ is a rotation domain. Let  $\tau_\theta$ be an 
irrational rotation on the unit disk $\triangle$, then for any annulus 
$T_r=\{r\leq |z|\leq 1\}$, the function given by $\displaystyle{P(z)=\frac{1}{|z|^2 mod(T_r)}}$ 
for $z\in T_r$ and $0$ for $z$ in $\triangle\setminus T_r$ defines a $\tau_\theta$-invariant 
probability measure. By using a parametrization of   $D$ and 
Proposition \ref{prop.Krengel}, we have that rotation domain cycles are 
conservative and strongly conservative. Moreover, by Theorem 1.6,  page 117 in
\cite{Krengel}, we have the following fact.

\begin{lemma}\label{lm.con.dis} [Almost everywhere convergence on the 
dissipative set] For every $f\geq 0$ in $L_1(\overline{\C})$, the series 
$$\sum_{n=0}^\infty |R^*|^n(f)$$ converges and is finite almost 
everywhere on the dissipative set.
\end{lemma}

\begin{proof}
It is enough to show that the series  $$\sum_{n=0}^\infty |R^*|^n(f)$$ 
converges on every wandering set $W$ of finite positive measure. In fact, we 
show that the series defines an integrable function. Since $W$ is wandering 
then by Proposition \ref{prop.bounds}, we have
$$\int_W \sum_{n=0}^\infty 
|R^*|^n(f)\leq \sum_{n=0}^\infty \int_{R^{-n}(W)} f=\int_{\bigcup 
R^{-n}(W)}f<\|f\|_1.$$\end{proof}

Using Proposition \ref{prop.Krengel}, we reformulate results of M. Lyubich and 
C. McMullen to obtain the following lemma.

\begin{lemma}\label{lemma.dichotomy}
Let $R$ be a rational map. Then
\begin{itemize}
\item Either $C(R)\cap J_R \subset P_R$ or $C(R)=\overline{\C}$.
\item Either $SC(R)\cap J_R\subset P_R$ or $SC(R)=\overline{\C}$.
\item In the case where $C(R)=\overline{\C}$ but $P_R\neq \overline{\C}$, then there exists a fixed 
point of the Beltrami operator supported on the Julia set if and only if  $R$ is a 
flexible Latt\`es map.
\end{itemize}

\end{lemma}

\begin{proof}
We follow ideas and arguments of Lyubich  and McMullen  (see \cite{LyuTypical} and  \S 
3.3 in  \cite{Mc1}).  First by Poincar\'e recurrence theorem for conservative 
actions (see \cite{AaronsonBook}) we have  $$\liminf_{n\rightarrow \infty} d(x,R^n(x))=0$$ for 
almost every $x\in C(R)$ where $d$ is the spherical metric.
Assume that the set $B=(C(R)\cap J_R)\setminus P_R$ has positive Lebesgue 
measure. Then $R(B)\subset B$, otherwise  if $B$ is not invariant then the set 
$X=B\cap R^{-1}(P_R)$ has positive Lebesgue measure. Then $R^n(X)\cap X$ has zero Lebesgue 
measure for $n\neq 1$ since $R^{n}(X)\subset P_R$ for all $n\geq 1$. This implies that the sets 
$\{R^{-k}(X)\}$ are pairwise almost disjoint, thus  $X$ is wandering, which contradicts that 
$B\subset C(R)$.  Then again by Poincar\'e recurrent theorem we have
$$\limsup_{n\rightarrow\infty} d(R^n(x),P_R)\geq d(x,P_R)>0$$ for almost every point in 
$B.$ Now using Koebe distortion 
arguments as in section 1.19 of \cite{L} or  \S 3.3 in \cite{Mc1}, we obtain that the 
closure  of any invariant positive Lebesgue measure subset $E$ in $B$ contains a disk 
and thus $\overline{E}=\overline{B}=\overline{\C}.$ 

In particular, this holds for $E=(SC(R)\cap J_R)\setminus P_R\subset B$ whenever $E$ has positive 
measure since, by Proposition \ref{prop.Krengel}, the set $E$ is invariant.

For the third part, by the hypothesis and Poincar\'e's theorem for almost every point $x\in C(R)\setminus P_R$ we have 
that $\limsup d(R^n(x),P_R)>0.$ Hence, if $J_R$ supports a non-zero invariant Beltrami differential then  
Theorem 3.17 in McMullen's book \cite{Mc1} finishes the proof.
\end{proof}

We also use the following proposition.

\begin{proposition}\label{prop.invariantmeasure} Let $R$ be a rational map and 
let $f \in L_1(\overline{\C})$ be a fixed point of the Ruelle operator $R^*$. Then 
there exists a fixed point of the Beltrami operator $\mu\in L_\infty(\overline{\C})$
such that $\int_\C f(z) \mu(z) |dz|^2 \neq 0$. In fact 
$|R^*||f|=|f|$ and $|f|$ defines an absolutely continuous finite invariant measure that 
satisfies $\frac{|f|}{f}=\mu$ almost everywhere on the support of $f.$
\end{proposition} 

\begin{proof}
This summarizes the results given in Lemma 11 and Corollary 12 in 
\cite{MakRuelle}. \end{proof}

Note that Lemma \ref{lemma.dichotomy} and Proposition 
\ref{prop.invariantmeasure} give necessary conditions for Sullivan's and 
Fatou's conjectures. 

The following lemma is a consequence of part 3 of Lemma 5  and part 1 
of Theorem 3 in \cite{MakRuelle}. For $v$ in $\C$,  define the function 
$\gamma_v$ by  $$\gamma_v(z)=\frac{v(v-1)}{z(z-1)(z-v)}.$$
Throughout the paper, the Ces\`aro averages $A_n(\gamma_v)$ of 
functions of the form $\gamma_v$,  
will be always taken with respect to the Ruelle operator $R^*.$

\begin{lemma}\label{lemma.fixed.Beltrami} Let $R$ be a rational map. 
\begin{enumerate}
\item If the Fatou set $F_R$ contains a periodic attracting domain $V$ then 
there exists an invariant Beltrami differential $\mu$ supported on the grand 
orbit of $V$ and a critical value $v_0$ such that $\int_\C \mu(z) 
\gamma_{v_0}(z)|dz|^2\neq 
0$.
\item  If $\mu$ is an invariant Beltrami differential in $L_\infty(J_R)$ then 
$\mu\neq 0$ if and only if there exists a critical value $v_0$ such that 
$\int_\C \mu(z) \gamma_{v_0}(z)|dz|^2\neq 0$.
\end{enumerate}
\end{lemma}

The following theorem gives a connection between Sullivan's conjecture and 
mean ergodicity on a suitable subspace of $L_1(\C)$.

\begin{theorem}\label{th.MeanErg.1}
Let $R$ be a rational map such that $SC(R)\cap P_R$ has Lebesgue measure zero. 
Then $R$ satisfies Sullivan's conjecture if and only if $R^*$ is mean-ergodic 
on $Hol(R)$ with the topology inherited from $L_1(J_R)$.
\end{theorem}

\begin{proof}
Assume that $R$ satisfies Sullivan's conjecture. Then either there is no 
invariant Beltrami differential supported on the Julia set or $R$ is a 
flexible Latt\`es map. 
If there is no invariant Beltrami differential supported on $J_R$ then 
the Beltrami operator on $L_\infty(J_R)$ does not have fixed points. Hence 
$(Id-R^*)(L_1(J_R))$ is dense in $L_1(J_R)$. Otherwise by 
the Hanh-Banach theorem, there would exist a non-zero continuous functional 
$\mathcal{L}$ on $L_1(J_R)$ such that $(Id-R^*)(L_1(J_R))\subset 
ker(\mathcal{L})$. By the Riesz representation theorem there exists $\alpha\in 
L_1(J_R)$ which is fixed by the Beltrami operator and representing the functional 
$\mathcal{L}$, which is a contradiction. Then $A_n(R^*)(f)$ converges to $0$ for 
all $f$ in $L_1(J_R).$ In particular, this happens when $f\in Hol(R)$. Thus 
$R^*$ is mean-ergodic. 

If $R$ is a flexible Latt\`es map then, since $R$ is postcritically finite, the 
space $A(S_R)$ is finite dimensional.  So $A(S_R)$ coincides with the 
subspace $Hol(R)$ and, by the mean ergodicity lemma, the Ruelle operator 
$R^*$ is mean-ergodic. Note that for this part of the proof we do not need that 
$SC(R)\cap P_R$ has Lebesgue measure $0.$

Reciprocally, assume that $R^*$ is mean-ergodic in $Hol(R)$. To finish the 
theorem is sufficient to show that if the Julia set $J_R$ supports an 
invariant Beltrami differential  then $R$ is a flexible Latt\`es map.
Indeed, by Lemma \ref{lemma.fixed.Beltrami}, every fixed point of the Beltrami operator 
supported on the Julia set defines a continuous invariant functional 
$\mathcal{L}$ on $L_1(J_R)$ which is non-zero on $Hol(R)$. Let $\phi$ in 
$Hol(R)$ be such that $\mathcal{L}(\phi)\neq 0$. By mean ergodicity 
$A_n(R^*)(\phi)$ converges to some non-zero element $f$ in $L_1(J_R)$. 
By Proposition \ref{prop.invariantmeasure}, $|f|$ defines an absolutely continuous 
finite invariant measure, hence $supp(f)\subset SC(R)$. Since $SC(R)\cap P_R$ has 
measure zero then by Lemma \ref{lemma.dichotomy},   
the map $R$ is a flexible Latt\`es map. \end{proof}

Let us note that Theorem \ref{th.MeanErg.1} holds even when $S_R=\emptyset.$
Moreover, when $P_R\neq J_R$, we can consider the space 
$A(S_R)$ instead of $Hol(R)$ and get the same conclusion as in the previous 
theorem. 
As mentioned in the introduction, conjecturally, for a map $R$ with non-empty 
Fatou set, the condition on 
the postcritical set of the previous theorem is always fulfilled. 
On the other hand the mean ergodicity of Ruelle operator is a rather simple 
consequence of geometric conditions of $S_R$. For example, 
the Ces\`aro averages $A_n(\gamma_v)$ are weakly convergent on measurable 
subsets $Y$ of $S_R$ of finite hyperbolic area (see \cite{CabMakHyp} and 
discussion therein). Hence, the existence of fixed points of the Ruelle 
operator is the main obstruction to extend Theorem \ref{th.MeanErg.1} in 
full generality.

In general, as the following corollary shows, the Ruelle operator for 
rational maps is not mean-ergodic on $A(S_R)$ or $L_1(\C)$.

\begin{corollary}
 Let $R$ be a rational map.
\begin{enumerate}
\item If $F_R$ contains an attracting periodic component  then 
$R^*:A(S_R)\rightarrow A(S_R)$ is not mean-ergodic.
 \item If $F_R$ contains a periodic non-rotational component then 
$R^*:L_1(\C)\rightarrow L_1(\C)$ is not mean-ergodic.

\end{enumerate}
\end{corollary}
\begin{proof}
Part (1). By contradiction. Let $V$ be an attracting domain of $F_R$ and assume 
that $R^*$ is mean-ergodic on $A(S_R)$. Then by part (1) of  Lemma 
\ref{lemma.fixed.Beltrami}, there exists a Beltrami differential $\mu$, 
supported on the grand orbit of $V$, such that the functional 
$l_\mu(\phi)=\int_\C
\mu(z) 
\phi(z)|dz|^2\neq 0$ on $A(S_R)$ and $l_\mu(R^*\phi)=l_\mu( \phi)$. By the mean 
ergodicity lemma, there exists $\psi\in A(S_R)$ with $R^*(\psi)=\psi$ such that 
$l_\mu(\psi)\neq 0.$ This implies that the restriction $f$ of $\psi$ on the 
grand orbit of $V$ is non-zero and it is a fixed point of $R^*.$ By Proposition 
\ref{prop.invariantmeasure}, the function $|f|$ defines a finite invariant measure. 
However, the grand orbit of $V$ belongs to the dissipative set. Hence 
using arguments similar to the previous theorem we can show that $f$ is $0$ 
almost everywhere on the grand orbit of $V$. This is a contradiction. 

Part (2). Under this hypothesis, it is enough to show, that the grand orbit 
of any non rotational periodic domain supports a non-zero invariant Beltrami 
differential. We conclude this part using the classification of Fatou components  
and  arguments similar to those of part (1). On part (1) we already considered the attracting 
case, so by the classification of Fatou components, we have to consider the cases where 
there is a Fatou component $V$ which is either parabolic or superattracting. 
In the parabolic case, let $\Phi$ be a linearisation function defined on the 
grand orbit of $V$. That is a function satisfying $\Phi(R(z))=\Phi(z)+1.$ Then 
a short computation shows that the function $\mu=\frac{\overline{\Phi}'}{\Phi'}$ gives a 
non-zero invariant Beltrami differential.
In the superattracting case, let $\Phi$ be now the B\"ottcher coordinate around 
a neighbourhood  $U$ of the superattracting  cycle. Let 
$\nu(z)=\frac{\Phi \overline{\Phi}'}{\overline{\Phi}\Phi'}$ then for $z\in U$, we have 
$\nu(R(z))\frac{\overline{R'}(z)}{R'(z)}=\nu(z)$. Using the dynamics of $R$ we can extend to a 
non-zero invariant Beltrami differential defined on the grand orbit of 
$U.$ \end{proof}

\textbf{Remark}. By the corollary above, in order to deal with mean ergodicity 
of the Ruelle operator on $A(S_R)$ we will often assume that $R$ does not 
accept invariant Beltrami differentials defining non-trivial quasiconformal 
deformations supported on the Fatou set.  According to Theorems 6.2 and 6.8 
in \cite{McMSull},  this is the same as saying that there 
are no Herman rings and if $c$ is a critical point in $F_R$, then either 
\begin{enumerate}
\item $\mathcal{O}_+(c)=\bigcup_n R^n(c)$ is finite, or 
\item $\mathcal{O}_+(c)$ accumulates to a parabolic point $p$ and there is no other 
critical point $\tilde{c}\in F_R$ with forward orbit $\mathcal{O}_+(\tilde{c})$ 
accumulating to $p$ and $\mathcal{O}_+(\tilde{c})\bigcap \mathcal{O}_+(c)=\emptyset$. 
\end{enumerate}
In  other words,  any non-zero invariant functional on the space 
$A(S_R)$ is induced by an $R^*$-invariant Beltrami differential supported on the 
Julia set. Indeed we have the following equivalent statement.
\begin{proposition}\label{prop.inducedinv}
 Any non-zero invariant functional on the space $A(S_R)$ is induced by an $R^*$-invariant 
 Beltrami differential supported on the Julia set if and only if the Ces\`aro averages $A_n(f)$ 
 converges to zero with respect to the $L_1$ norm over the Fatou set for every $f\in A(S_R)$.
\end{proposition}
\begin{proof}
 Let $A(F_R)$ be the space of holomorphic Lebesgue integrable functions 
on $F_R.$ Then $R^*$ acts on $A(F_R)$ with $\|R^*\|\leq 1.$

We prove the proposition by contradiction. If there exists $f\in A(S_R)$ such that $\int_{F_R}|A_n(f)|$ does not 
 converges to zero, then by the separation principle on the space $A(F_R)$ there exist a 
 non-zero $R^*$-invariant continuous functional $l$ with $l(f|_{F_R})\neq 0.$ 
 
 On the other hand, by the Hahn-Banach extension and Riesz representation 
 theorems there exists a function $\nu\in L_\infty(F_R)$ representing 
 $l$ on $A(F_R)$, that is $l(g)=\int_{F_R}g \nu |dz|^2$ for every $g\in A(F_R).$ 
 We claim that we can choose $\nu$ to be an invariant Beltrami differential 
 supported on $F_R$. Indeed, if $\nu$ is not invariant then take any $*$-weak 
 accumulation point of the sequence $\nu_n=\frac{1}{n}\sum_{i=0}^{n-1} B^{i}(\nu)$ 
 where $B$ is the Beltrami  operator. Then by the mean ergodicity lemma $\nu_\infty\in L_\infty(F_R)$
 is an invariant Beltrami differential supported on the Fatou set. Moreover $\nu_\infty$ defines
 the same functional on $A(F_R)$ as $\nu.$
 
Indeed, consider a subsequence $n_i$  such that $\nu_{n_i}\rightarrow \nu_\infty$ in the $*$-weak 
topology on $L_\infty(F_R)$  so we have
 $$\int_{F_R} g(z)\nu_\infty(z)|dz|^2=\lim_{i\rightarrow \infty}\int_{F_R} \nu_{n_i}(z)g(z)|dz|^2$$
 $$=\lim_{i\rightarrow \infty}\int_{F_R} \nu(z) A_{n_i}(g(z))|dz|^2$$
 $$=\lim_{i\rightarrow \infty} \frac{1}{n_i}\sum_{j=0}^{n_i-1}l(R^{*j}(g))=l(g).$$
 Hence $L(g)=l(g|_{F_R})$ for $g\in A(S_R)$ gives a non-zero continuous $R^*$-invariant functional 
 on $A(S_R)$  induced by an invariant Beltrami differential supported on $F_R$ as claimed.

 To get a contradiction we need to show that there is no Beltrami differential  
$\mu$ supported on $J_R$ such that 
 $$L(g)=\int_\C\mu(z) g(z) |dz|^2=\int_\C\nu(z) g(z) |dz|^2$$ 
 for any $g\in A(S_R).$
 
Assume that there is such a $\mu$. Let $\phi(a)=\int_\C \gamma_a(z)(\nu(z)-\mu(z))|dz|
^2$ be the potential function for the invariant differential $\nu-\mu$, then $\phi$ is 
a continuous function on $\C$. Since $\gamma_a\in A(S_R)$ for $a\in P_R$, then $\phi(a)=0$. Now we follow the 
proof of Theorem 3 in \cite{MakRuelle}. By invariance we have $\phi(R(a))=R'(a)\phi(a)$, 
which implies that $\phi(a)=0$ for every repelling periodic point $a$, and hence on $J_R$. 
Therefore $\overline{\partial}\phi=\nu-\mu=0$ almost everywhere in $J_R.$ Which is a 
contradiction. \end{proof}

Later on, we will discuss the relation between the topologies on $Hol(R)$ 
induced by $L_1$ norms over the Fatou and the Julia set, respectively.

Next we give some conditions under which the Ruelle operator does not have a 
fixed point. We call an integrable function $f$ \textit{regular} if  the derivative
$\overline{\partial}f$, taken in the sense of distributions, is a finite complex valued 
measure.  Examples of non-regular functions are given by characteristic 
functions of suitable compact sets. 
See the remark after the proof of the following theorem.
\begin{theorem}\label{thm.regularfixed}
Let $R$ be a rational map. Assume that the postcritical set $P_R$ is
such that either

\begin{itemize}
 \item the diameter with respect to the spherical metric of all components $D$ of $S_R$ are uniformly 
bounded away from $0$, or
 
\item $J_R\cap P_R$  is contained in the union of the boundaries of the components   of  $S_R$. 
\end{itemize}

\noindent Then  $R^*$ has a regular non-zero fixed  point  if and only if $R$ is a flexible Latt\`es map.
\end{theorem}

\begin{proof}
Assume that $R$ is a flexible Latt\`es map. Then $R$ has an invariant Beltrami differential 
$\mu$ unique up to multiplication by scalars  (see Milnor 
\cite{MilLat}).  
This differential $\mu$ defines a non-zero functional $l_\mu$ on $A(S_R)$ given 
by the pairing
$$
l_\mu(\phi)=\int_\C \phi(z) \mu(z)|dz|^2.
$$ 
Since $A(S_R)$ is finite dimensional, then $R^*$ is mean-ergodic on $A(S_R)$, 
and by the separation principle there exists a non-zero fixed point $f\in 
A(S_R)$ of the Ruelle operator, which is unique up to multiplication by 
scalars. Since $f$ is an integrable holomorphic function outside finitely many 
points of $\C$, the function $f$ is rational with simple poles only. Hence the 
distributional derivative $\overline{\partial}f$ is a finite combination of Dirac 
measures supported on the poles of $f.$

The Beltrami differential $\mu$ is a unique fixed point of the 
Beltrami operator. Since the Beltrami operator $B:L_1(\C)\rightarrow L_1(\C)$
is dual to $R^*$, then by the separation principle, the operator $R^*$ is mean-ergodic
on $L_1(\C).$ Hence $f\in L_1(\C)$ is a unique fixed point of the Ruelle operator up to 
scalar multiplication.

Now let $f$ be a non-zero regular fixed point of the Ruelle operator. Then by
Proposition \ref{prop.invariantmeasure}, the function  $|f|$ is the density of a finite 
invariant measure and there is  an invariant Beltrami differential $\mu$ with 
$\mu=\frac{|f |}{f}$ on $supp(f)$. Hence $supp(f)\subset SC(R)$ by Proposition \ref{prop.Krengel}.   Then by
Lemma \ref{lemma.dichotomy}, either $R$ is a flexible Latt\`es map or the 
support $supp(f)$ is contained in the postcritical set.
 
Without loss of generality we may assume that $P_R$ is a proper subset of $\C$. Let  
$\nu$ be such that $d\nu=\overline{\partial} f$ and set
$$
F(z)=\int_\C \frac{d\nu(t)}{t-z}.
$$ 
Since the support of $\nu$ belongs to the support of $f$, the map $F(z)$ is 
holomorphic outside $P_R.$ Note that $F(z)=f(z)$ holds for Lebesgue 
almost every point. Indeed, by Weyl's lemma there exists an entire function $h(z)$ such that
$h(z)=F(z)-f(z)$ almost everywhere. Since $f(z)$ has compact support then $F(z)$ 
converges to $0$ as $z$ converges to $\infty$, thus $h(z)=0$.  In particular, $F(z)$ is 
identically $0$ outside the support of $\nu$.

We claim that the first condition of Theorem \ref{thm.regularfixed} implies that $\nu$ is identically $0$ on $\C$.  
Recall that the generalised Mergelyan's theorem (see, for example, 
Theorem 10.4 on Gamelin's paper \cite{GamelinUniAlg}) states:  \textit{if 
the diameters of all  components of the complement of a compact set $K$ 
on the plane $\C$ are bounded away from $0$ then any continuous function 
which is holomorphic on the interior of $K$ is a uniform limit of rational 
functions with poles outside of $K$}. Since $P_R$ satisfies the conditions 
of the generalised Mergelyan's theorem and has empty interior, then 
any continuous function on $P_R$ is a uniform limit of rational functions 
with poles outside $P_R.$ Given a finite set of complex numbers $b_i$ 
and points $a_i$ in $\C \setminus P_R$, define $r(z)=\sum \frac{b_i}{z-a_i}$ 
and we get
$$\int r d\nu=\sum b_i F(a_i)=0.
$$ Hence $\nu$ represents a zero functional on the space of 
continuous functions on $P_R.$ By  the Riesz representation theorem the 
measure $\nu$ is null as claimed. 

Thus $f(z)=0$ almost everywhere which is a contradiction.  

Now we assume that $P_R$ satisfies the second condition and follow closely the 
arguments of part 3 of Proposition 14 in \cite{MakRuelle}. Let $\{Y_i\}$ be the family of 
all components of 
$\C\setminus P_R$. Then we claim that $f(z)=0$ almost everywhere on 
$\cup_i \partial Y_i\subset P_R$ whenever $F(z)=f(z)$ almost everywhere. 

Otherwise, there exists a component $Y_{i_0}$ and $E\subset \partial 
Y_{i_0}$ with $m(E)>0$ and $\int_E f\neq 0$. Since $P_R$ is compact 
we can assume that $\infty$ belongs to  $Y_{i_0}$ by conjugating by a M\"obius map. 
Then the function $F_E(z)=\int_E \frac{|dt|^2}{t-z}$
is a continuous function on the plane which is holomorphic outside $\partial 
Y_{i_0}.$ Again, by the generalised Mergelyan theorem $F_E$ is a uniform limit 
of rational functions with poles in $Y_{i_0}$. Hence using similar 
arguments as above we obtain $\int F_E(z) \overline{\partial} f(z)=0$. Applying Fubini's 
theorem we compute   $$0=\int F_E(z) \overline{\partial} f(z) =\int 
\overline{\partial}f(z)\int_E \frac{|dt|^2}{t-z}$$$$=\int_E |dt|^2\int \frac{\overline{\partial}f(z)}{t-z}.$$ 
$$=-\int_EF(t)|dt|^2=-\int_E f(t)|dt|^2.$$

This is a contradiction, so we have the claim.

Now if $P_R\subset \bigcup \partial Y_i$ then  by the claim $f(z)=0$ almost 
everywhere in $\C$ this contradiction finishes the proof. \end{proof}

Let us note that if $P_R\cap J_R\subset  \partial V$, where  $V$ is a component of
the Fatou set $F_R$ then the second condition of the theorem is always satisfied. Indeed,
in this situation only finitely many components of $F_R$ contain $P_R\cap J_R$ on its 
boundary and hence $P_R\cap J_R$ belongs to the boundary of finitely many components 
of $S_R.$ So, if $R$ has a completely invariant Fatou component then $R$ satisfies the 
second  condition of the theorem. On the other hand, we do not know an example of a 
rational map $R$ such that $S_R$ has infinitely many components.
On the discussion above we saw that the convergence of Ces\`aro averages on 
subspaces of $L_1(\C)$ is closely related to the existence of non-trivial 
invariant Beltrami differentials under some conditions. 
\medskip

\textbf{Remark}: The arguments of the proof of Theorem \ref{thm.regularfixed} also provide explicit examples of compact sets 
whose characteristic functions are not regular. Indeed, let $K$ be a compact subset 
satisfying the generalized Mergelyan theorem (for example a positive Lebesgue measure 
Cantor set), if the characteristic function $\chi_K$ is regular, by Weyl's lemma we have 
$\chi_K(x)=\int \frac{\overline{\partial}\chi_K}{z-x}$ almost everywhere which 
contradicts the generalized Mergelyan theorem.

\begin{corollary}\label{cor.totvariation}
Let $R$ be a rational map satisfying the conditions of  Theorem \ref{thm.regularfixed}. Suppose there 
exists a critical value $v\in V(R)$ such that the total variation 
of $\overline{\partial}A_n(\gamma_v)$ is uniformly bounded. Then  $R$ is not 
structurally stable.
\end{corollary}

\begin{proof}

Since the total variation of the sequence $\overline{\partial}A_n(\gamma_v)$ is 
uniformly bounded, it is $*$-weakly precompact 
when acting on continuous functions. 

Let $m_0$ be a complex valued measure 
which is a $*$-weak accumulation point of this sequence. Since 
$supp(\overline{\partial}A_n(\gamma_v))\subset P_R$ then $supp(m_0)\subset P_R$. 
Considering $\overline{\partial}A_n(\gamma_v)$ as measures,  a straightforward computation 
shows that $$A_n(\gamma_n)(z_0)=-\int_{\C} 
\frac{\overline{\partial}A_n(\gamma_v)(t)}{t-z_0}=-\int_{P_R} 
\frac{\overline{\partial}A_n(\gamma_v)(t)}{t-z_0}$$ for every $z_0$ outside $P_R.$ 

If $m_0\neq0$ then, as in Theorem \ref{thm.regularfixed}, by the generalized Mergelyan theorem the integral $-\int_\C \frac{dm_0(t)}{t-z}$ is 
non-zero and is an accumulation point of $A_n(\gamma_v)$ in the topology of 
pointwise convergence on $S_R$. Therefore, this integral is a regular non-zero fixed point. 
By Theorem \ref{thm.regularfixed}, the map $R$ is a Latt\`es map which is not structurally stable.

The remaining the case is when $\overline{\partial}A_n(\gamma_v)$ converges 
to $0$ in the $*$-weak topology. Let $\mu\in L_\infty(\C)$ and consider its
potential function $$F_\mu(z)=-\frac{z(z-1)}{\pi}\int  \frac{\mu(\zeta) 
|d\zeta|^2}{\zeta(\zeta -1)(\zeta-z)}.$$ Recall that $F_\mu$ is continuous on $\C$ and satisfies $\overline{\partial}F_\mu(z)=\mu(z)$ 
in the sense of distributions.

We claim that if $\mu$ is a fixed point of the Beltrami operator then we have 
 $\int_\C \mu(z) \gamma_v(z)|dz|^2=0.$
 
 Indeed, since $\int F_\mu(z) 
\overline{\partial}A_n(\gamma_v)(z)$ converges to $0$ and $\mu$ is 
invariant, we have that 
$$
\int_\C F_\mu(z)
\overline{\partial} A_n(\gamma_v)(z)=-\int_\C 
\overline{\partial}F_\mu(z) A_n(\gamma_v)(z)|dz|^2$$$$=-\int_\C \mu(z) 
A_n(\gamma_v)(z)|dz|^2
$$ which by duality and the invariance of $\mu$ implies 
$$
\int_\C \mu(z) A_n(\gamma_v)(z)|dz|^2=\int_\C \mu(z) 
\gamma_v(z)|dz|^2=0
$$
as claimed.

If $R$ is structurally stable and $\mu$ is an invariant differential then  by part (3) of Lemma 5 
in \cite{MakRuelle}, the potential $F_\mu$ satisfies  for  $a\in \C$ the equation
\begin{equation*}F_\mu(R(a))-R'(a)F_\mu(a)=-R'(a)\sum_{c_i} \frac{1}{R''(c_i)}F_\mu(R(c_i))\gamma_a(c_i)  \tag{*}\end{equation*} 
where the sum is taken over all critical points $c_i$ with $i=1...(2deg(R)-2).$
Moreover, by Theorem 3 of \cite{MakRuelle} there exists a $(2deg(R)-2)$ dimensional space $X$ 
of invariant Beltrami differentials, so that the correspondence 
$$\beta:\mu\mapsto F_\mu(R(a))-R'(a)F_\mu(a)$$ is a linear isomorphism of $X$ onto its image. 
But by the claim $$F_\mu(v)=\int_\C \mu(z) \gamma_v(z)|dz|^2=0$$ for every invariant Beltrami 
differential $\mu$, in particular for $\mu\in X.$ Since $v$ is a critical value, then using  the right 
part of equation $(*)$, we get that the space $\beta(X)$ has dimension at most $2deg(R)-3$, which 
is a contradiction. \end{proof}

 As an immediate corollary we have the following.

\begin{corollary}
Assume the Julia set $J_R$ has positive Lebesgue measure, then $J_R$  does not 
supports a non-zero invariant Beltrami differential 
if and only if, for any critical value $v$, the sequence 
$\overline{\partial}A_n(\gamma_v)$ converges to $0$ on every 
continuous function $\phi$ on $J_R$ with distributional 
derivative $\phi_{\overline{z}}\in L_\infty(J_R)$.
\end{corollary}

\begin{proof}
 By contradiction. First note that for every continuous function $\phi$ with distributional derivative $\phi_{\overline{z}}$ we have
 $$\int \phi \overline{\partial}A_n(\gamma_v)=-\int \phi_{\overline{z}} A_n(\gamma_v(z))|dz|^2.$$
 By duality $$-\int \phi_{\overline{z}} A_n(\gamma_v(z))|dz|^2=-\frac{1}{n}\sum_{j=1}^ n \int B^j(\phi_{\overline{z}})\gamma_v(z)|dz|^2,$$ where $B$ is the Beltrami operator.
 If there exists $\phi$ such that the sequence $\int \phi \overline{\partial}A_n(\gamma_v)$
 does not converges to $0$, then any  accumulation point of $-\frac{1}{n}\sum_{j=1}^ n B^j(\phi_{\overline{z}})$ is a non-zero invariant Beltrami differential. 
 Reciprocally, if there is a non-zero invariant Beltrami differential $\mu$, then 
 its potential $F_\mu$ is continuous and $(F_\mu)_{\overline{z}}=\mu$, then $$-F_\mu(v)=\int \mu(z) A_n(\gamma_v(z))|dz|^2=\int F_\mu \overline{\partial}A_n(\gamma_v)$$
 converges to $0$. So $-F_\mu(v)=0$ for all critical value $v$. This contradicts part 2 of Lemma \ref{lemma.fixed.Beltrami}.
\end{proof}

In the following statements we show that there are no fixed points of the 
Ruelle operator among the examples of non-regular functions mentioned above.

\begin{proposition}\label{prop.chiA}
Assume that a function $g=f+c\chi_A$, where $c$ is a constant and $\chi_A$ is 
the characteristic  function of a measurable set $A$ such that 
$A\setminus supp(f)$ has positive measure. If $g$ is a fixed point of 
the Ruelle operator then $R^*(f)=f$ and $c=0.$
\end{proposition}

\begin{proof}

If $B=A\setminus supp(f)$, then by Proposition \ref{prop.Krengel} and 
Proposition \ref{prop.invariantmeasure}, we have that $D=B\cap SC(R)$ has positive measure and 
$\mu(z)=\frac{|g(z)|}{g(z)}=\mu(R(z))\frac{\overline{R'(z)}}{R'(z)}$ almost 
everywhere on $supp(g).$
Then there exists $k$ such that the measure of $D\cap R^{k}(D)$ is positive, 
and hence the set $C=D\cap (R^{k})^{-1}(D)$ has positive measure.  Moreover, 
we have $\mu=\frac{|c|}{c}$ on $C.$ If $c\neq 0$ then by invariance of $\mu$, we 
have that $(R^k)'$ is real valued on $C$ and thus $C\subset 
((R^{k})')^{-1}(\mathbb{R})$ which contradicts $m(C)>0.$ 

\end{proof}
As an immediate corollary we have that a simple function cannot be a fixed point 
of the Ruelle operator.

\begin{corollary}\label{cor.simplefunc}
 If $f=\sum c_i \chi_{A_i}$ where the $A_i$ are distinct measurable 
sets, then $f$ is a fixed point of the Ruelle operator if and only if $c_i=0$ 
for all $i.$
\end{corollary}

Let us show that for any infinite closed set $K$ in $\overline{\C}$, the space $H(K)$
always contains a non-regular function, even in the case when $K$ has zero 
Lebesgue measure. On the other hand $H(K)$ contains characteristic functions if 
and only if $K$ has positive measure.  Nevertheless,  by Lemma \ref{lemma.dichotomy}, Proposition \ref{prop.invariantmeasure} and Theorem \ref{th.MeanErg.1}, if the postcritical 
set has measure zero then any fixed point of the Ruelle operator is necessarily a 
regular function. In general we conjecture that any fixed point of the Ruelle 
operator is necessarily a regular function. In the last section we will discuss 
the existence of fixed points when the postcritical set has positive measure.

\begin{proposition}\label{prop.nonregular}
 Let $K$ be an infinite closed subset of $\overline{\C}$, then $H(K)$ contains 
 non-regular functions.
\end{proposition}

\begin{proof}
It is enough to show the proposition when $K$ is a bounded infinite closed subset of $\C$.
Otherwise we can compose with a M\"obius map. For $\mu\in L_\infty(\C)$ the assignment
$\mu\mapsto F_\mu|_K$ defines a continuous compact operator $S:L_\infty(\C)\rightarrow C(K)$ (see Theorem 7 of Chapter 3 page 56 of \cite{GardLakic}). Then the dual operator $S^*:M(K)\rightarrow L^*_\infty(\C)$ defined 
on the space $M(K)$ of all finite complex valued measures with total variation as norm is compact also. 
By direct computation we have that if $\nu\in M(K)$ then 
$$S^*(\nu)(z)=\int \gamma_a(z)d\nu(a).$$ Hence by Fubini's theorem 
the image of $S^*$ belongs to $H(K).$ If $U\in M(K)$ is the closed unit ball, then 
$S^*(U)$ is closed in $H(K)$.  

If  every element in $H(K)$ is regular so $H(K)=\bigcup_{n\geq 0}(S^*(nU))$ then by Baire's 
category theorem there exists  $n_0$ such that $S^*(n_0U)$ has non-empty interior. But, by the
compactness of $S^*$ any ball inside $S^* (n_0U)$ is finitely dimensional which implies that 
$H(K)$ has finite dimension and hence $K$ is finite. This is a contradiction. \end{proof}

We endow the space $Hol(R)$ with two non-complete norms. 
The first is given by $|f|_1=\int_{F_R} |f|$
and the second by $|f|_2=\int_{J_R} |f|$ for $f\in Hol(R)$.
Let us call $Hol_1$ and $Hol_2$ the respective normed spaces.  If the Fatou
set is empty then $Hol_1=0.$ Similarly, $Hol_2=0$ when  the Julia set  has measure $0$.
The operator $R^*$ is a contraction on each space. 
 
Next we show that any compatibility between these two 
topologies on $Hol(R)$ gives rise to a sort of rigidity of the dynamics of 
$R$. 

\begin{proposition}\label{prop.contId}
Fix a rational map $R$ with $F_R\neq \emptyset$. 
 
\begin{enumerate}
\item If the identity map $Id:Hol_1 \to Hol_2$ is
continuous, then there are no fixed points of the 
Beltrami operator on the Julia set $J_R$.

\item If the Fatou set $F_R$ admits an invariant Beltrami differential 
defining a non-zero functional on $Hol_1$, then $Id:Hol_2 \to Hol_1$ is 
continuous if and only if $J_R$ has Lebesgue measure zero. 
\end{enumerate}
\end{proposition} 
\begin{proof}
For the first part, suppose that there is a non-zero fixed point $\mu$ 
of the Beltrami operator on the Julia  set. 
Then $\mu$ defines a non-trivial invariant continuous linear functional $l_\mu$ 
on $Hol_2$. 
Since $Id$ is a continuous operator, the functional 
$l(\alpha)=l_\mu(Id(\alpha))$ is a continuous invariant linear 
functional on $Hol_1$ and hence extends to the completion 
$\overline{Hol_1}\subset L_1(F_R)$ of $Hol_1$ with respect to its norm. 
By the Hahn-Banach and Riesz representation theorems 
there exists a function $\nu\in L_\infty(F_R)$ such that  
$$
l(\phi)=\int _{F_R} \phi(z) \nu(z) |dz|^2
$$
for all $\phi \in \overline{Hol_1}$.   

As in the proof of Proposition \ref{prop.inducedinv}, we can assume  that  $\nu$ is a fixed point 
of the Beltrami operator supported on the Fatou set. 

In particular, for $\gamma_a\in Hol(R)$ with $a\in 
P_R,$ the continuous functions $F_{\nu_\infty}(a)=\int_\C 
\gamma_a(z)\nu_\infty(z) 
|dz|^2$ and 
$F_{\mu}(a)=\int_\C \gamma_a(z)\mu(z) |dz|^2$ coincide
on the orbit of all critical values, and hence on $P_R.$ This however 
contradicts Theorem 3 of \cite{MakRuelle} using an analogous argument to the proof of 
Proposition \ref{prop.inducedinv}. 

For the second part notice that if the Julia set has measure zero, 
then $Hol_2$ consists only of the $0$ function. 
So, assume that the Lebesgue measure of the Julia set is not zero. 
Let $\mu$ a fixed point of Beltrami operator with 
$$
l_\mu(\phi)=\int_{F_R} \phi(z) \mu(z)|dz|^2 
\neq 0.
$$
Applying the same arguments as in the first part, we get a 
continuous linear functional on $Hol_2$ defined by a fixed point of  
Beltrami operator now supported on the Julia set. 
This again by an analogous argument
of Proposition \ref{prop.inducedinv} gives a contradiction, hence 
the Lebesgue measure of $J_R$ is zero. 
\end{proof}

Now define $X=(Id-R^*)(Hol(R))$ and let $X_1$ and $X_2$ be the 
closures of $X$ in the spaces of $Hol_1$ and $Hol_2$, respectively. 
The proof of Theorem \ref{th.MeanErg.1} shows that 
if there is no invariant Beltrami differential supported on a Julia set of 
positive measure, then we get $X_2=Hol_2$ which coincides with $Hol(R)$ as a set and hence also 
$X_1\subset X_2$. We will prove the converse in Proposition 
\ref{pr.fixedpointBelt} below. First we need a technical result. 

\begin{lemma}\label{lemma.linearfunc} 
Let $l$ be a  linear functional on $Hol(R)$. If $ X_1 \subset ker(l)$ then 
$l$ is continuous on $Hol_1$. 
\end{lemma}

\begin{proof}
Let $W$ be the finite dimensional space of all linear combinations of the functions
$\gamma_v$, with $v$ a critical value of $R$. 
We will show first that $Hol_1$ equals the sum $X_1 + W$; note that $X_1\cap 
W$ may be non-zero. Indeed,  by definition the space $Hol_1$ is the linear 
span of $\gamma_a(z)$ where $a$ is an element in the union of the forward 
orbits of all critical values. For every critical value $v$ of $R$, by Lemma 5 
of \cite{MakRuelle}, we have 
$$
R^*(\gamma_v)=\frac{1}{R'(v)}\gamma_{R(v)}+w
$$ 
for some $w$ in $W$. Since $R^*(\gamma_v)-\gamma_v\in X_1$ and 
$-w+\gamma_v \in W$, the element $\gamma_{R(v)}$ 
is a sum of elements in $X_1$ and $W$. But $X_1$ is 
invariant under $R^*$, so by an induction argument we conclude that
$\gamma_{v_n}\in X_1+ W$ for every $v_n=R^n(v)$. 
Therefore the space $X_1$ has finite codimension on $Hol_1$.  

If $X_1\subset \ker(l)$ then $l$ projects to a linear functional 
${\cal L}$ defined on the finite dimensional space $Hol_1 /X_1$. 
As $X_1$ is closed this implies that both ${\cal L}$ and $l$ are 
continuous.\end{proof}

\begin{proposition} \label{pr.fixedpointBelt}
Suppose that the Julia set $J_R\neq \overline{\C}$ and has positive Lebesgue 
measure. Then,  the only invariant  Beltrami differential supported on $J_R$ is zero
if and only if $X_1\subset X_2$.
\end{proposition}

\begin{proof}
If there is no invariant Beltrami differential supported on the Julia set, 
then, as in Theorem \ref{th.MeanErg.1},  $X_2=Hol_2$. Since the Lebesgue measure of $J_R$ is positive then  $Hol_2$ coincides as a set with $Hol(R)$. In other words, every element $\omega \in Hol(R)$ can be approximated by elements of the form $\alpha_i-R^*(\alpha_i)$ with $\alpha_i\in Hol(R)$ in the $L_1(J_R)$ and hence $X_2$ contains $X_1$. 

On the other direction, assume $X_1\subset X_2$. As in Lemma 
\ref{lemma.linearfunc}, we have $Hol(R)=X_1+F$ for some 
finite dimensional vector space $F\subset W$. For an invariant Beltrami 
differential $\mu$ supported on the Julia set, the assignment 
$l_\mu(\phi)=\int_{J_R} \phi(z) \mu(z)|dz|^2$ defines an invariant 
linear continuous 
functional on $Hol_2$ for which we have $X_2\subset ker(l_\mu)$.   
Lemma \ref{lemma.linearfunc} shows that $l_\mu$ is 
continuous on $Hol_1$, so by the Riesz representation theorem 
there is a function $\nu\subset L_\infty(F_R)$ so that 
$l_\mu(\phi)=\int_{F_R}\phi(z)  \nu(z) |dz|^2$. 
However, since $\mu$ is invariant, using analogous arguments 
to the proof of Proposition \ref{prop.contId}, we again get a contradiction 
to Theorem 3 of \cite{MakRuelle}. \end{proof}

\section{Action on $L_p$ spaces}

If we want to extend the theory to $L_p(K)$ spaces (with $K$ completely 
invariant), we need to modify somehow the operators in an \textit{ad-hoc} 
manner. In the formulas below, $d$  denotes the degree of $R$. 
Let $p,q$ be such that $1/p+1/q=1$. 

The action of $R$ by \textit{pull-back} on $L_p$ is given 
by 
$$
R_{*p}(\phi)=\frac{1}{\sqrt[p]{d}}(\phi\circ R)|R'|^{\frac{2}{p}}\frac{R'}{\overline 
{R'(z)}}, 
$$
when $\phi \in L_p$. 

Similarly, the \textit{push-forward} action of $R$ on $L_q$ (for $1< q < 
\infty$)
is defined by
$$ 
R^{*}_{q}(\phi)=\frac{1}{\sqrt[p]{d}}\sum 
\phi(\zeta_i)\frac{\zeta_i'}{\overline{\zeta_i '}}|\zeta_i'|^{\frac{2}{q}},
$$ 
where the sum is taken over all branches $\zeta_i$ of $R^{-1}$, 
that is, they satisfy $R(\zeta_i(z))=z$ for almost all $z$. 
The constants are suitably chosen so that $R^{*}_{p}$ and $R_{*q}$ are 
mutually dual. Indeed, if $\phi\in L_p$ then for any $g\in L_q$, we have $$\int 
R_{*p}(\phi(z))\overline{g}(z)\,|dz|^2=\frac{1}{\sqrt[p]{d}}\int 
\phi(R(z))|R'(z)|^{\frac{2}{p}}\frac{\overline{R'(z)}}{R'(z)}\overline{g}(z)\,
|dz|^2$$ 
and, after changing variables with $R(z)=t$, is equal to 
$$\frac{1}{\sqrt[p]{d}}\int 
\phi(t)
\left[\sum_{i=0}^d 
\left(\overline{g}|R'|^{\frac{2}{p}}\frac{\overline{R'}}{R'}\right)
(\zeta_i(t))|\zeta'_i(t)|^2\right] |dt|^2$$ for all branches 
$\zeta_i$ of 
$R^{-1}.$ By direct calculation, for every $i$ we have 
$$\left(|R'|^{\frac{p}{2}} 
\frac{\overline{R'}}{R'}\right)(\zeta_i)|\zeta'_i|^2=|\zeta'_i|^\frac{2}{q} 
\frac{\zeta'_i}{\overline{\zeta_i'}} .$$  
Hence the previous expression is equal to $$\int \phi(t) 
\overline{\left[R^*_q(g)\right](t)}|dt|^2.$$ 
Moreover, we have $R^*_p\circ R_{*p}=Id$ on $L_p$. 

We have now two continuous families of contractions 
depending on $p >0$, which are mutually dual for 
$p\ge1$ and includes the Ruelle and Beltrami 
operators, for $p=1$ and $p=\infty$, respectively.  

The next theorem is the $L_p$ version of the action 
of Ruelle operators (compare with the previous section). 
Unfortunately, for $1<p< \infty$,  the Ruelle operator on $L_p$
does not detect whether there is an invariant Beltrami differential without 
an associated fixed point for the Ruelle operator.

\begin{theorem}\label{thm.lpcase}
Let $K$ be a completely invariant set of positive measure. Given $p$  with $1<p<\infty$, then the operator $R^*_p$  has a fixed point in $L_p(K)$ 
if and only if $R$ is a flexible Latt\`es map.
\end{theorem}
\begin{proof}
If $R^*_p$ has a fixed point then its dual $R_{*q}$ has 
a fixed point on $L_q(K)$. Let $\psi$ be a  fixed point of $R_{*q}$ 
then $\frac{|\psi|}{\psi}$ is an invariant Beltrami differential. 
On the other hand, the function  $f=|\psi|^q$ is integrable and 
satisfies 
$$
f(z)=\frac{f(R(z))|R(z)'|^2}{deg(R)}.
$$ 
In this case $supp(f)$ is completely invariant. By Proposition 
\ref{prop.Krengel} and the discussion afterwards, either $R$ 
is a Latt\`es map or $supp(f) \subset C(R)\cap P_R\subset J_R$. But the latter 
is not possible. Indeed,   let $E$ be the operator $E:C(J_R)\rightarrow C(J_R)$ 
on the space of complex valued continuous functions $\phi$ on $J_R$ 
given by  $$E_R(\phi)(z)= \frac{1}{deg(R)}\sum_i \phi(\zeta_i(z))$$ where 
the sum is taken over all branches $\zeta_i$ of $R^{-1}.$

Let us recall Lyubich's Theorem 5 
in \cite{LyuTypical} which states that every continuous functional invariant with 
respect to $E$ is induced by a multiple of the maximal entropy measure. 

Let $\nu(z)$ be the measure such that $d\nu(z)=f(z) \, |dz|^2$. 
Now, 
$$\int_{J_R} E_R(\phi)d\nu=\frac{1}{deg(R)}\int_{J_R} \left[\sum_i 
\phi(\zeta_i(z))\right] f(z)|dz|^2.$$
By the arguments and computations of Proposition \ref{prop.bounds}, the latter 
is equal to
$$\frac{1}{deg(R)}\int_{J_R} \phi(z)f(R(z))|R'(z)|^2|dz|^2 
=\int_{J_R}\phi d\nu.$$ Then $\nu$ is a multiple of the maximal entropy 
measure. 
By Zdunik's Theorem (see \cite{ZdunikParab}) $R$ is a postcritically finite 
rational map. Hence $f(z)=0$ almost everywhere since 
$supp(f)\subset C(R)\cap P_R$ which is a contradiction. Thus $R$ is a flexible 
Latt\`es map. 

Conversely, if the map $R$ is a flexible Latt\`es map, then there 
exist an integrable function $f_0$ such that $f_0=\frac{f_0(R)(R')^2}{deg 
R}$ and 
$\mu(z)=\frac{\overline{f_0(z)}}{|f_0(z)|}$ is a fixed point of the
Beltrami operator. 
Therefore $\psi=|f_0|^{\frac{1}{p}}\mu$ is a fixed point for $R_{*p}$ on 
$L_p(K)$ and induces a fixed point of the Ruelle operator in $L_p(K)$. 
Since the dual operator satisfies $R^*_q\circ R_{* q}=Id$, the converse follows 
immediately.
\end{proof}
 Note that the theorem above shows that if $R^*_{p_0}$ has a non-zero fixed point in $L_{p_0}(K)$ for some $p_0$, then $R^*_p$ has a fixed point in $L_p(K)$ for all $p$ with $1<p< \infty.$
\section{Fixed points of bidual actions and uniform ergodicity}

In this section, we start with  the following lemma which is a summary 
of results due to H. P. Lotz in \cite{Lotz}.

\begin{lemma}\label{Lotz.lemma}
Let $B$ be a Grothendieck space with the Dunford-Pettis property and let 
$S:B\rightarrow B$ be a power-bounded operator. Then
\begin{enumerate}
 \item If $S$ is mean-ergodic, then $S$ is uniformly ergodic.
 \item If the space $Fix(S^*)$ of fixed points of the dual of $S$ is separable 
then $S$ is uniformly ergodic and the Ces\`aro averages uniformly converge to 
a compact projection.
\end{enumerate}
 
\end{lemma}

Part (1) is the content of Theorem 5 in \cite{Lotz}. Part (2) is the content of 
Theorem 7 in \cite{Lotz}.

For the rational map $R$, let $T:B(S_R)\rightarrow B(S_R)$  be the Thurston operator as 
defined in Section 3.  We have the following theorem.

\begin{theorem}\label{th.meanergosep} 
Let $R$ be a rational map such that $P_R$ is not the 
whole sphere,   the conservative set $C(R)$ does not contain any Fatou  component and 
 $1$ belongs to the spectrum  $\sigma(R^*)$ on $A(S_R)$. 
Then the following four conditions are equivalent. 
\begin{enumerate}
\item The space $Fix(T^*)$ is separable.
\item The operator $T$ is mean-ergodic.
\item The Ruelle operator $R^*$ is uniformly ergodic.
\item The map $R$ is a flexible Latt\`es map. 
\end{enumerate}
\end{theorem}
 
\begin{proof} By the discussion on Section \ref{sec.Thurs}, the space $B(S_R)$ 
is a Grothendieck space with the Dunford-Pettis property. Then Lemma 
\ref{Lotz.lemma} applied to the operator $T$ gives the implications from (1) to (2) and from (2) to (3).

To show (3) implies (4), note that since $R^*$ is a contraction then $R^*$ is a 
power-bounded operator. By the uniform ergodicity lemma and the assumption,  
the value $1$ is an isolated eigenvalue of $\sigma(R^*)$. 
Then the Ruelle operator $R^*$ has a non-zero fixed point $\phi$ in $A(S_R)$. 
Hence by Proposition \ref{prop.invariantmeasure}, the modulus $|\phi|$
defines an invariant finite measure such that the support of $|\phi|$ contains a component of  $S_R$. Then by Proposition \ref{prop.Krengel} and Lemma \ref{lemma.dichotomy}, the map $R$ is a flexible Latt\`es map. 

Now we show (4) implies (1). For a flexible Latt\`es map, the space $A(S_R)$ is 
finite dimensional, and hence $A^*(S_R)$ and its dual are finitely 
dimensional too, which implies (1). \end{proof}

The condition that $1$ belongs to the spectrum is necessary for the 
discussion around Sullivan's conjecture. Otherwise, the Beltrami operator does 
not have fixed points. However, in this situation there are open questions: 
\begin{enumerate}
\item Does there exists a rational map $R$ with infinite 
postcritical set such 
that the norm of $R^*$ on $A(S_R)$ is strictly smaller than 1? Or more generally: 

\item Is it true that if $1$ does not belong to the spectrum of 
$R^*$ on $A(S_R)$ then $R$ is postcritically finite?  
\end{enumerate}
By the uniform ergodicity lemma, the conditions of the questions above imply 
that $R^*$ is uniformly ergodic on $A(S_R).$

The following corollary gives a partial answer to these questions for the class 
of $J$-stable rational maps.  

\begin{corollary}\label{cor.uniformergod}
If $R$ is $J$-stable such that  $P_R\neq J_R$ then the following are equivalent. 
\begin{enumerate}

\item The operator $R^*$ is uniformly ergodic on $A(S_R)$. 
\item The map $R$ is hyperbolic and postcritically finite. 
\end{enumerate}
 
\end{corollary}
\begin{proof}
Condition (2) implies (1) since, in this case, $A(S_R)$ is finite dimensional. 

For the converse, first note that $1$ does not belong to $\sigma(R^*)$. 
Indeed, if $1$ belongs to $\sigma(R^*)$ then, by Theorem \ref{th.meanergosep}, 
$R$ is a flexible Latt\`es map which contradicts  $J$-stability.
Hence every invariant Beltrami differential supported on $S_R$ defines a
trivial quasiconformal deformation. By Lemma \ref{lemma.fixed.Beltrami}, any 
non-zero invariant Beltrami differential supported on the Julia set defines a 
non-trivial quasiconformal deformation.  Hence by Theorem E in Ma\~n\'e, Sad, 
Sullivan (see \cite{MSS}), the map $R$ is hyperbolic. Moreover, by Theorem D 
in \cite{MSS}, we have that $R$ is postcritically finite.\end{proof}

If the Julia set $J_R$ has positive measure and does not support non-zero 
invariant Beltrami differentials, then the action of the Ruelle operator on 
$L_1(J_R)$ is mean-ergodic because $\lim \frac{1}{n} \sum_{i=0}^{n-1} B^i(\mu)=0$ for
all $\mu \in L_\infty(J_R)$ by the mean ergodicity lemma. Thus,
by duality the Ces\`aro averages of the Ruelle operator converge 
to $0$ in the weak topology and hence in the strong topology by 
the mean ergodicity lemma. In contrast with this fact, we show that the action 
of $R^*$ on $L_1(J_R)$ is uniformly ergodic only in the case when $J_R$ has  
Lebesgue measure zero. In slightly more generality, we prove the following.

\begin{theorem}\label{th.uniformergod} 
Let $R$ be a rational map. If $K$ is a completely invariant set of positive Lebesgue measure,  then the Ruelle operator is not uniformly ergodic on $L_1(K)$. 
\end{theorem}

\begin{proof}

By contradiction. If the measure of $K$ is positive and $R^*$ is uniformly ergodic in 
$L_1(K)$, we claim that $R^*$ is an automorphism of $L_1(K)$. 
It is enough to show that $R^*$ is injective since $R^*$ is surjective 
by the relation $R^*\circ R_*=Id$. 

For every $\lambda$ with $|\lambda|<1$ and 
$\phi$ non-zero element in $L_1(K)$, the element 
$$
\phi_\lambda=\sum_{n=0}^\infty \lambda^n R^n_*(\phi)\not\equiv 0 
$$ 
and satisfies $\lambda R_*(\phi_\lambda)=\phi_\lambda-\phi$. If $\phi$ is a non-zero element in $Ker(R^*)$ then  
$R^*(\phi_\lambda)=\lambda\phi_\lambda$. 
It follows that $1$ is not an isolated eigenvalue in the spectrum,  
which contradicts the uniform ergodicity of $R^*$. 
Therefore $R^*$ is an isomorphism as we claimed.
This, in turn, implies 
$$
R_*\circ R^*=Id
$$ 
on $L_1(K)$. 
Now let $x_1$ and $x_2$ be different fixed points of $R$, and fix 
a point $b$ different from $x_1$ and $x_2$. 
Then the restriction of 
$\omega_b(z)=\frac{(x_1-b)(x_2-b)}{(z-x_1)(z-x_2)(z-b)}$ 
to $K$ is integrable. 
Since $\omega_b$ is rational, and the equation $R_*\circ 
R^*(\omega_b)=\omega_b$ holds almost everywhere  
on a set of positive measure, then it holds on the whole Riemann sphere.  
If we take $b$ such that neither $b$, $R(b)$ nor $R^{-1}(b)$ are critical 
values of $R$, then $R^*(\omega_b)$ has a non-trivial pole on $R(b)$, 
which implies that $R_*\circ R^*(\omega_b)$  has poles in $R^{-1}(R(b))$ which 
are different to the  poles of $\omega_b$. 
This  contradiction finishes the proof.  
\end{proof}

We have the following corollary which complements Corollary \ref{cor.uniformergod}
in the case when $P_R=\overline{\C}$. In this case, since $S_R=\emptyset$, we consider the uniform ergodicity over 
$Hol(R)$ equipped with the $L_1$ topology on $\C$. 

\begin{corollary}
Let $R$ be a rational map with $P_R=\overline{\C}$ then  $R^*$ is not uniformly ergodic on 
$Hol(R)$ equipped with the topology of $L_1(\C)$.
\end{corollary}

\begin{proof}
 Since $P_R=\overline{\C}$ then $Hol(R)$ is everywhere dense in $L_1(\C)$. If $R^*$ is uniformly ergodic in $Hol(R)$, it is also uniformly ergodic on $L_1(\C)$ which contradicts
 Theorem \ref{th.uniformergod}.
\end{proof}

Let us observe that the Thurston operator $T$ leaves $B_0(S_R)$ invariant. 
Since the Ruelle operator is continuous in the $*$-weak topology on
$B_0(S_R)$ then $R^*$ is dual to the restriction of $T$ to $B_0(S_R)$, 
so we have $(T|_{B_0(S_R)})^{**}=T$.

By the facts discussed on Section \ref{sec.Thurs},
every endomorphism $Q$ on $B_0(S_R)$ is either compact or 
there exists an infinite dimensional subspace $E$ of $B_0(S_R)$ such that the 
restriction $Q|_E:E \rightarrow B_0(S_R)$ is an isomorphism onto its image. 
Let us consider $Q=Id-T$ and investigate two extremal situations. 
First when $Q:B_0(S_R)\to B_0(S_R)$ is an isomorphism onto its image, that is 
the space $E$ is the whole of $B_0(S_R)$. Second, when $Q$ is a compact 
operator. 
Our goal is to show that  these two extremal situations imply the uniform 
ergodicity of the operator $T$ which is equivalent to the uniform ergodicity of 
$R^*$.

\begin{theorem}\label{th.Qcompact}
If $Q=Id-T$ is either an isomorphism onto its image or a compact operator, 
then $T$ is uniformly ergodic.  
\end{theorem}

\begin{proof}
If $Q:B_0(S_R) \to B_0(S_R)$ is an isomorphism onto its image, 
then the subspace $(Id-T)B_0(S_R)$ is closed, hence $T$ is uniformly 
ergodic by the uniform ergodicity lemma. 

Now assume $Q$ is compact. As a consequence of the spectral decomposition 
theorem, for any $\epsilon>0$ there is a splitting $B_0(S_R)=F_\epsilon \oplus 
X_\epsilon$ in which both subspaces, $F_\epsilon$ and $X_\epsilon$, are 
invariant under $Q$, with $dim(X_\epsilon)<\infty$ and such that the norm of 
$Q$ restricted to $F_\epsilon$ is less than $\epsilon$.   
As $F_\epsilon $ and $X_\epsilon$ are $Q$-invariant these are also 
$T$-invariant. 
Consider the restriction $Q|_{F_\epsilon}$. 
Since the norm on $F_\epsilon$ is small, it follows that $(Id-Q)|_{F_\epsilon}$ 
is invertible. 
The von Neumann series yields $(Id-Q)^{-1}=\sum_{n=0}^\infty Q^n$, at least on 
$F_\epsilon$. 
This means that 
$(Id-Q)^{-1}|_{F_\epsilon}$ is a compact operator 
as it is the norm limit of compact operators.
We have $Id-Q=T$, 
by above $T$ is invertible on $F_\epsilon$ and its inverse $T^{-1}|_{F_\epsilon}$ is a compact operator. Hence $T|_{F_\epsilon}$ is a compact isomorphism
which implies that $F_\epsilon$ has finite dimension. 
Therefore $B_0(S_R)$ has  finite dimension too since $X_\epsilon$ is finite dimensional and 
$T$ is a contraction over a finite dimensional space, thus is uniformly ergodic by the uniform ergodicity lemma.  \end{proof}

As an immediate corollary we have the following. 

\begin{corollary}
Under the conditions of the last theorem, 
the operator $Q$ is compact if and only if $R$ is postcritically finite.
\end{corollary}

\begin{proof}
If $R$ is postcritically finite, then $B_0(S_R)$ has finite dimension.
Since $Q$ is continuous, then $Q$ is compact. 
Reciprocally, if $Q$ is compact, as in the proof of the previous theorem, 
the linear space $A(S_R)$, which is dual to $B_0(S_R)$,
has finite dimension. This is only possible when $P_R$ is finite. 
\end{proof}

The following is the main theorem of this section.

\begin{theorem}\label{th.ThurstonMeanergodic}
Assume that the conservative set $C(R)$ does not intersect the Fatou set $F_R$. 
Then the Thurston operator $T$ is mean-ergodic on $B_0(S_R)$. Moreover, $T$ has 
a non-zero fixed point in $B_0(S_R)$ if and only if $R$ is a flexible Latt\`es 
map.
\end{theorem}
\begin{proof}
Assume that there is no non-zero fixed point of $R^*:B^*_0(S_R)\to B^*_0(S_R)$, 
then by arguments similar to those of the proof of Theorem \ref{th.MeanErg.1} 
we have that $(I-T)(B_0(S_R))$ is dense in $B_0(S_R)$. This implies that for 
$\phi\in B_0(S_R)$ the averages $\frac{1}{n}\sum_{i=0}^{n-1}T^i(\phi)$ converge 
to $0.$ Then $T$ is mean-ergodic in $B_0(S_R)$.

If $R^*$ has a fixed point, by
Proposition \ref{prop.invariantmeasure}, Lemma \ref{lemma.dichotomy} and 
the arguments presented in the proof of Theorem \ref{th.meanergosep} 
it follows that $R$ is a flexible Latt\`es map and $T$ is mean-ergodic on 
$B(S_R)$ and hence on $B_0(S_R).$

Finally, by the separation principle of Section \ref{section4}, a mean-ergodic 
operator $T$ has a non-trivial fixed point if and only if its dual $R^*$ has a 
non trivial fixed point. By the arguments above this is the case precisely when 
$R$ is a flexible Latt\`es map.
\end{proof}

The following corollary shows that if $T$ has a fixed point in $B(S_R)$ close 
enough to $B_0(S_R)$ in the Teichm\"uller distance, with respect to the 
Teichm\"uller norm,  then $R$ is a Latt\`es map. 

\begin{corollary}\label{cor.Teichflexible}
Let $R$ be as in Theorem \ref{th.ThurstonMeanergodic}. 
Let $\alpha\in B(S_R)$ with $T(\alpha)=\alpha$ and $\|\alpha\|_T=1$ where 
$\|\cdot\|_T$  is the Teichm\"uller norm. 
If 
$$
dist(\alpha,B_0(S_R))=\inf_{\phi \in B_0(S_R)}\|\alpha -\phi\|_T<1, 
$$
then $R$ is a flexible Latt\`es map.
\end{corollary}
\begin{proof}
Given the hypothesis, we can find $\phi_0\in B_0(S_R)$ subject to 
$\|\alpha-\phi_0\|_T<1$. 
By Theorem \ref{th.ThurstonMeanergodic}, $T$ is mean-ergodic. 
If $T$ has a fixed point in $B_0(S_R)$ then  $R$ is a Latt\`es map. 
Otherwise the Ces\`aro averages $A_n(\phi)$ converge to $0$ 
for every $\phi$ in $B_0(S_R)$. 
In particular, we have 
$$
1=\|\alpha\|_T=\lim \|A_n(\alpha - \phi_0)\|_T \leq \|\alpha-\phi_0\|_T<1,
$$ 
which is a contradiction. 
Thus $T$ has a fixed point in $B_0(S_R)$ and $R$ is a Latt\`es map. \end{proof}

We say that $R$ is \textit{dissipative} 
if the dissipative set of $R$ is the whole Riemann sphere; in other words, the 
conservative set has Lebesgue measure zero. 

\begin{theorem}\label{th.dissipative}
Let $R$ be a dissipative map. 
Then, for any $\alpha$ in $B_0(S_R)$, the orbit under
the Thurston operator $T^n(\alpha)$ converges weakly to $0$. 
\end{theorem}
 
\begin{proof}
Since the conservative set of $R$ has Lebesgue measure zero, 
by Lemma \ref{lm.con.dis},  we have that $\sum_{n} R^{*n}(\phi)(z)$ absolutely 
converges almost everywhere in $\C$ for every $\phi$ in $A(S_R)$.
 
In particular, it follows that $R^{*n}\phi$ converges to $0$ almost everywhere 
on $\C$. But  $R^{*n}\phi$ is a bounded sequence in $A(S_R)$. Therefore 
$R^{*n}\phi$ defines a normal family of holomorphic functions on $S_R$. Hence   
$R^{*n}\phi$ converges pointwise to $0.$ By duality, the $T$ orbit of any 
element in $B_0(S_R)$ weakly converges to $0$. 
\end{proof}

An operator satisfying the conclusion of the previous theorem is called 
\textit{weakly asymptotic}.

\begin{theorem}\label{th.kukareko}
Suppose that the measure of the postcritical set $P_R$ is 
$0$ and that the image of the closed unit ball in $B_0(S_R)$ under $Id-T$ is closed, 
then  $R$ satisfies Sullivan's conjecture.
\end{theorem}
\begin{proof} 
We use Theorem 2.3 and Theorem 3.3  in \cite{FonfLinRubinov}, which state 
that, for a separable Banach space $X$ and power-bounded operator 
$T$, if the image of the closed unit ball under $(Id-T)$ is closed then either $T$ is 
uniformly ergodic or the space $\overline{(Id-T)(X)}$ contains an isomorphic 
copy of an infinite dimensional dual Banach space. 

Now let us show that $T$ is uniformly ergodic on $B_0(S_R)$. By contradiction, 
if $T$ is not uniformly ergodic then the space $\overline{(Id-T)(B_0(S_R))}$ 
contains an infinite dimensional dual Banach space $Y$. By Theorem 1 in 
\cite{BonetWolf}, 
$B_0(S_R)$ is isomorphic to a closed subspace of $c_0$. On the other hand, any 
closed infinite dimensional subspace of $c_0$ contains a closed subspace 
isomorphic to $c_0.$ Hence $Y$ contains an isomorphic copy of $c_0.$ Let $h: 
c_0\rightarrow Y$ be this isomorphism onto the image. 
Since $Y$ is a dual space, the closed unit ball in $Y$ is $*$-weak compact, so the 
isomorphism $h$ can be extended to a continuous linear operator $O:\ell_\infty
\rightarrow Y$.  Then $O^*:Y^*\rightarrow (\ell_\infty)^*$  is continuous with 
respect to the $*$-weak topology. As $Y$ is separable any bounded set of 
linear functionals contains a subsequence converging in the $*$-weak topology. 
Since $\ell_\infty$ is a Grothendieck space   $O^*$ is a 
weakly compact operator.  Thus $O$ and hence  $h$ are weakly compact. But a 
separable Banach space with weakly compact unit ball is reflexive, this 
contradicts the fact that $c_0$ is not reflexive. Therefore $T$ is uniformly 
ergodic on $B_0(S_R).$

If $T$ has a fixed point in $B(S_R)$, then $Id-T$ is not invertible in 
$B(S_R)$. Then $(Id-T):B_0(S_R)\rightarrow B_0(S_R)$ is not an isomorphism. 
Thus by the uniform ergodicity lemma the value $1$ is an isolated eigenvalue 
for $T$ in $B_0(S_R)$. Then the value $1$ is an isolated eigenvalue for $R^*$ in $A(S_R)$. Hence $R^*$ is mean-ergodic on $A(S_R)$. Therefore
by hypothesis $R^*$ is mean-ergodic in $Hol(R)$ with topology inherited 
by $L_1(J_R)$. Then Theorem \ref{th.MeanErg.1} finishes the proof.

\end{proof}

Let us add some comments about the closeness of the image  of 
the closed unit ball under $(Id-T)$. Since $B(S_R)$ is a dual space, the image 
of its closed unit ball under $(Id-T)$ is always closed. When the image of the closed unit 
ball of $B_0(S_R)$ under $(Id-T)$ is not closed, consider the space 
$$X=cl((Id-T)^{-1}(B_0(S_R))\subset B(S_R),$$ then again on $X$, the image of 
the closed unit ball under $(Id-T)$ is closed. But $X$ is invariant under $T$ 
and $B_0(S_R)\oplus Fix(T,B(S_R)) \subset X.$  Then we have the following 
result. 
\begin{proposition} Let $X$ be as in the discussion above. 
 Assume that $X$ satisfies one of the following conditions: 
\begin{enumerate} 
 \item $X$ is separable.
\item $X$ is a Grothendieck space.
\item The operator $T:X\rightarrow X$ is mean-ergodic.
\end{enumerate}
Then $T:B(S_R)\rightarrow B(S_R)$ is uniformly ergodic.
\end{proposition}

\begin{proof}
 If $X$ is separable, since the image of the closed unit ball under $(Id-T)$ is closed 
and $\overline{(Id-T)(X)}=B_0(S_R)$ then, as in the proof of Theorem \ref{th.kukareko},  either  $T$ is 
uniformly ergodic or $B_0(S_R)$ contains an isomorphic copy of 
infinite dimensional dual Banach space. But as the proof of Theorem \ref{th.kukareko} shows, $B_0(S_R)$ can not contain such a copy. Hence $T$ is uniformly ergodic. 

If $X$ is a Grothendieck space, then $(Id-T):X\rightarrow B_0(S_R)$ is 
weakly compact. It follows that $(Id-T):B_0(S_R)\rightarrow B_0(S_R)$ is 
also weakly compact. Recall that every endomorphism $Q$ of $B_0(S_R)$ is either compact or there exists an infinite dimensional subspace $Y$, isomorphic to $c_0$, such that the restriction of $Q$ on $Y$ is an isomorphism onto its image. We conclude that 
$(Id-T)$ is a compact endomorphism of $B_0(S_R)$, and by duality $(Id-T)$ is also compact on $B(S_R)$.  Thus $T$ is uniformly ergodic by Theorem \ref{th.Qcompact}.

Finally, if $T:X\rightarrow X$ is mean-ergodic then Proposition 3.8 in 
\cite{FonfLinRubinov} states precisely that the image of the closed unit ball in 
$B_0(S_R)$ under $(Id-T)$ is closed and the arguments of Theorem 
\ref{th.kukareko} apply again. \end{proof}

In the next section  we analyse further properties of $X$ in a more general 
setting. 

\section{Hamilton-Krushkal sequences}

In this section, to avoid cumbersome calculations and definitions, we consider 
rational maps $R$ satisfying 
two conditions:
\begin{itemize}
\item First, there is no non-trivial quasiconformal deformation supported 
on the Fatou set. That is, by definition, every invariant Beltrami differential supported
on the Fatou set defines a zero functional on $A(S_R)$. Other characterizations are given by Proposition \ref{prop.inducedinv}.
\end{itemize} 

As mentioned on the introduction, this condition does not impose any restriction on the study of Sullivan's 
conjecture if the Julia set is connected. 
\begin{itemize}
\item Second, the postcritical set $P_R$ does not support an absolutely 
continuous invariant measure with respect to the Lebesgue measure.
\end{itemize}

In other words, the measure of the intersection of $P_R$ with the strongly 
conservative set $SC(R)$ is zero. In the last section we will discuss and give 
partial results for the case where $P_R\cap SC(R)$ has positive measure. 

Let us consider the elements $\gamma_v(z)$ with $v$ a critical value. 
Let $D=\{A_n(\gamma_{v_i})\}$ be the set of Ces\`aro averages for all 
$\gamma_{v_i}$ with $v_i$ in the critical value set $V(R)$. 
A sequence $\{\phi_k\}$ in $A(S_R)$ is called \textit{degenerating, 
non-normalized}, if there are constants $C$ and $\epsilon>0$ with $\epsilon 
<\|\phi_k\|<C$ such that $\phi_k$ converges to $0$ pointwise.

By Bers' isomorphism theorem, we identify $A^*(S_R)$ with the space 
$B(S_R)$ and consider the seminorm on $B(S_R)$ given by
$$
K(l)=\sup \limsup_k(|l(\phi_k(z))|)
$$
where the supremum is taken over all sequences 
$\{\phi_k\}$ in $D$.

Note that a sequence in $D$ is either degenerating or precompact in norm. 
Indeed, if the sequence $\{A_n(\gamma_v)\}$ is not degenerating and does not 
converge in norm, then there exists a subsequence  which converges pointwise 
and 
locally uniformly to a non-zero limit which is fixed by the Ruelle operator. By 
Lemma \ref{lemma.dichotomy} and Proposition 
\ref{prop.invariantmeasure}, the map $R$ is a Latt\`es map.  By Theorem 
\ref{th.meanergosep}, the Ruelle operator $R^*$ is uniformly ergodic. Then $A_n(\gamma_v)$ converges in norm for every critical value $v$ which is a contradiction.
  
The \textit{Hamilton-Krushkal space} $HK(R)$ is the zero set of $K$.  
Since we have $K(l)\leq \|l\|$ on $B(S_R)$, the space $HK(R)$ is a closed 
subspace of $B(S_R)$.  

A subspace $Y$ of a Banach space $X$ is  called \textit{coseparable} whenever $X/Y$ is 
separable.

\begin{theorem}\label{th.coseparable}
A rational map $R$ satisfies Sullivan's conjecture if and only if $HK(R)$ is 
coseparable in $A^*(S_R)$. 
\end{theorem}
\begin{proof}
If there is no invariant Beltrami differential supported on the Julia set 
then, by Theorem \ref{th.MeanErg.1}, the action of $R^*$ on $L_1(J_R)$ is mean-ergodic and the Ces\`aro 
averages converge to $0$ in $L_1(J_R)$.  Thus by Proposition \ref{prop.inducedinv}, we have that $R^*$ is mean-ergodic on $A(S_R)$. 
So we get $HK(R)=A^*(S_R)$ and that the quotient is separable. 
If $R$ is Latt\`es, then $A^*(S_R)$ is finite dimensional and so is $HK(R)$. 

Conversely, assume that $HK(R)$ is coseparable. Then there exist a countable 
set $\{\alpha_i\}$ of elements in $A^*(S_R)$ such that $S=\{\alpha_i\}+ 
HK(R)$ is dense in $A^*(S_R)$. By induction, and a diagonal argument, we can 
pick a sequence $\{n_k\}$ such that for every $i$ and every critical value $v$ the sequence $\alpha_i(A_{n_k}(\gamma_v))$ converges.
Since $R^*$ is a contraction and $S$ is an everywhere dense subset 
of $A^*(S_R)$, the Ces\`aro averages $A_{n_k}(\gamma_v)$ converges weakly for 
every $v$. 
By the mean ergodicity lemma, the sequence $A_{n_k}(\gamma_v)$ 
converges in norm for all $v$. Let $\mu$ be a non-zero invariant Beltrami 
differential. 
By Lemma \ref{lemma.fixed.Beltrami} there exists a critical value $v_0$ such 
that 
$\int_\C  \mu(z) \gamma_{v_0}(z)\neq 0$. 
Thus the limit $f_0=\lim A_{n_k}(\gamma_{v_0})$ is a non-zero fixed point of 
the Ruelle operator in $A(S_R)$. Since the measure of $P_R\cap SC(R)$ is 
zero and $SC(R)$ contains $supp(f)$ then $SC(R)=\overline{\C}$ by part $(2)$ of Lemma \ref{lemma.dichotomy}. But $\mu(z)$ not $0$ almost everywhere, then by part (3) of Lemma \ref{lemma.dichotomy}  the map $R$ is a flexible Latt\`es map. 
\end{proof}

From the arguments in the previous proof, we have the following corollary. 

\begin{corollary}\label{cor.RuellemeanErg}
If $HK(R)$ is coseparable then the Ruelle operator $R^*$ is mean-ergodic on 
$A(S_R)$ equipped with the norm inherited from $L_1(J_R).$
\end{corollary}

Also, we have the following.

\begin{corollary}\label{cor.coseparable} 
The  space $HK(R)$ is coseparable if and only if 
$codim(HK(R))=0$. 
\end{corollary}

\section{Amenability and compactness}

In this section our goal is to give compactness 
conditions for suitable operators under which a map $R$ satisfies the Sullivan 
conjecture. We will keep the technical assumptions given at the beginning of 
the previous section.

 For every critical value  $v$  of $R$ define the operator 
$E_v:B(S_R)\rightarrow \ell_\infty$ by the formula 
$$
E_v(\psi)=\left( \int_{\C} \lambda^{-2}(z) 
\overline{\psi}(z)A_{n}(\gamma_v)(z)|dz|^2\right )_{n=0}^{\infty}.
$$

In particular,  an element $\phi \in A^*(S_R)$ belongs to $HK(R)$ 
whenever  $E_v(\phi)\in c_0$ for every critical value $v$. On the image of 
$E_v$,  the Thurston operator $T$ acts as 

$$
\hat{T}_v(E_v(\psi))= E_v(T(\psi)).
$$
This formula defines $\hat{T}_v$ as a linear endomorphism, not necessarily 
continuous, of the image of $E_v$. 

A \textit{mean} $m$ on $\ell_\infty$ is a positive linear functional which 
satisfies three conditions: 
\begin{itemize}
\item $m(1,1,1,...)=1$, 

\item  If $\sigma$ is the shift $\sigma(a_1,a_2,...)=(a_2,a_3,...)$, 
then $m(x)=m(\sigma(x))$ for any $x\in \ell_\infty$.  

\item $\liminf |a_i|\leq  |m(a_1,a_2,...)|\leq \limsup|a_i|$. 
\end{itemize}

A mean is also known as a \textit{Banach limit} on $\ell_\infty$. 

A linear operator $O:X\to X$ (not necessarily continuous) defined on a linear 
subspace $X$ (not necessarily closed) of $\ell_\infty$ has an 
invariant mean if there is a mean on $\ell_\infty$ with non-zero restriction to 
$X$ that satisfies $m(O(\alpha))=m(\alpha)$ for $\alpha\in X$. 
We denote by $M(O)$ the set of all invariant means for $O$. 

\begin{lemma}\label{lm.invmeanempty}
The set $M(\hat{T}_v)$ is empty if and only if 
$E_v(A^*(S_R))$ consists exclusively of sequences that converge to $0$. 
\end{lemma}
\begin{proof}
By definition means are invariant under the shift and bounded by the supremum 
of the elements of the sequence, this implies that $c_0\subset ker(m)$ for 
every 
mean $m$. Hence, if  $E_v(A^*(S_R))\subset c_0$, so we get 
$M(\hat{T}_v)=\emptyset$. 

Conversely, if there is an element $h\in A^*(S_R)$ with $E_v(h) \in 
E_v(A^*(S_R))\setminus c_0$ then there is a subsequence $\{n_j\}$ such that 
$h(A_{n_j}(\gamma_v))$ converges to a non-zero number $a$. 
By duality this implies that $A^*_{n_j}(h)$ allows a subsequence which 
converges $*$-weakly to a non-zero element $l_0\in A^*(S_R)$ that satisfies $T(l_0)=l_0$. 
Then $E_v(l_0)=(l_0(\gamma_v), l_0(\gamma_v),...)$ and, 
since $E_v(A^*(S_R))$ is a subspace of $\ell_\infty$, 
we conclude that $E_v(A^*(S_R))$ contains the constant sequence $1$. 
This implies that $E_v(A^*(S_R))$ intersects the space of convergent sequences 
in a non empty set. 
On convergent sequences, the functional $l:\{c_i\}\mapsto \lim c_i$ is 
continuous. 
By the Banach limit theorem there exists an extension ${\cal 
L}$ to all of $\ell_\infty$ which is a mean. 

Next we show that ${\cal L}$ is $\hat{T}_v$-invariant on $E_v(A^*(S_R))$.  
In fact,  we have 
$$
|T(h)(A_n(\gamma_v)) - h(A_{n+1}(\gamma_v))|\leq \frac{4\|h\|\|\gamma_v\|}{n},
$$  
for $h\in A^*(S_R)$. So the difference 
$$
|T(h)(A_n(\gamma_v)) - h(A_{n+1}(\gamma_v))|
$$
converges uniformly to $0$ as $n$ tends to $\infty$ on any ball of 
$A^*(S_R)$, 
and hence
$(\sigma -\hat{T}_v) E_v(A^*(S_R))\subset c_0.$ Thus $(\sigma -\hat{T}_v) 
E_v(A^*(S_R))$ belongs to $ker({\cal L})$. 
The invariance of ${\cal L}$ with respect to $\sigma$ implies the invariance 
of ${\cal L}$ with respect to $\hat{T}_v$. Therefore $M(\hat{T}_v)$ is non 
empty. \end{proof}

Next we show the finiteness of $M(\hat{T}_v)$ in very special cases.

\begin{theorem}\label{th.finiteBanach}
The set $M(\hat{T}_v)$ is finite if and only if $M(\hat{T}_v)$ contains at most 
one element or, equivalently,  if and only if $E_v(A^*(S_R))$ consists 
exclusively of convergent sequences. 
\end{theorem}

\begin{proof}
The first equivalence is clear after one notices that $M(\hat{T}_v)$ is convex, 
so we just worry about the second. 

If $M(\hat{T}_v)$ is empty, by Lemma \ref{lm.invmeanempty} we are done. 
Otherwise, again by the arguments on Lemma \ref{lm.invmeanempty}, if 
$E_v(A^*(S_R))\subset c$ then the only invariant mean is given by the 
restriction of the limit functional, $\{a_n\}\mapsto \lim a_n.$  
Reciprocally, we assume that there is only one invariant mean $m$ and set 
$X=\overline{E_v(A^*(S_R))}$. 
Again, by the same arguments given in Lemma \ref{lm.invmeanempty},  
the space $E_v(A^*(S_R))$ contains the element $(1,1,1,...)$.

Now we claim that if $L$ is a mean then  $L(x)=m(x)$ for all $x\in X$.
Indeed, by Lemma \ref{lm.invmeanempty}, we have that  $(\sigma -\hat{T}_v) 
E_v(A^*(S_R))\subset c_0.$ It follows that $L(\hat{T}_v(x))=L(x)$ for all $x\in 
E_v(A^*(S_R))$. Since $(1,1,1,...)\in E_v(A^*(S_R))$ then 
$L|_{E_v(A^*(S_R))}\neq 0.$ By uniqueness $L(x)=m(x)$ on $E_v(A^*(S_R))$, so 
by the continuity of $L$ and $m$ we get our claim. In particular,   
$E_v(A^*(S_R))\cap ker(m)\subset \bigcap_L ker(L)$ where the intersection is 
taken over all means $L.$

The space $X$ admits the decomposition $X=\C \cdot (1,1,1...) \oplus 
(ker(m)\cap X)$,  
and there is a  fixed point of the Beltrami operator $\mu$ such that 
$E_v^{-1}(X)=\C \mu \oplus 
E_v^{-1}(ker(m))$. By the Banach limit theorem (see Theorem 
4.1 in \cite{Krengel}) and the claim above, for every $h\in E_v^{-1}(ker(m))$ we get 
$$
\lim_{k\rightarrow \infty} \frac{1}{k}\sum_{j=0}^{k-1} h(A_j(\gamma_v))=0.
$$
But we have 
$$ 
\frac{1}{k}\sum_{j=0}^{k-1} 
h(A_j(\gamma_v))=h(\frac{1}{k}\sum_{j=0}^{k-1} 
A_j(\gamma_v)), 
$$ 
and thus, the sequence $f_k=\frac{1}{k}\sum_{j=0}^{k-1} A_j(\gamma_v)$ 
is a weakly convergent sequence of integrable holomorphic functions. Let $\phi_0$ be a weak 
limit of $f_k$, then $\phi_0\in A(S_R)$ and $\{f_k\}$ is a bounded sequence which, in particular, converges to $\phi_0$ in $*$-weak topology, hence $\phi_0$ is a pointwise limit of $\{f_k\}$.
Since $\mu$ is invariant, and $\int \mu \gamma_v \neq 0$ then $\int 
\mu \phi_0\neq 0.$ Now recall that the sequence  $\{A_n(\gamma_v)\}$ is either 
degenerating or converges by norm. If $\{A_n(\gamma_v)\}$ is degenerating then 
$f_k$ is also degenerating. This implies that $\phi_0=0$, which contradicts the 
existence of $\mu.$ Thus $\{A_n(\gamma_v)\}$ converges in norm. Then $E_v(h)$ 
is a convergent sequence for every $h\in A^*(S_R).$ \end{proof}

We have the following

\begin{theorem}\label{th.Mtfinite} 

The map $R$ is satisfies Sullivan's conjecture if and only if $M(\hat{T}_v)$ 
is finite for every critical value $v$.
\end{theorem}
\begin{proof}
If $M(\hat{T}_v)$ is finite then, by Theorem \ref{th.finiteBanach}, it consists 
of at most one element $m_v$. The space $E_v(A^*(S_R))$ 
consists only of convergent sequences for every critical value $v$. By 
definition $HK(R)$ contains  all elements $h$ in $A^*(S_R)$ such that $E_v(h)$ 
is a sequence converging to $0$ for every critical value $v.$ In other words,
$\cap_v ker(m_v(E_v(A^*(S_R))))\subset HK(R)$. But the space $\cap_v 
ker(m_v(E_v(A^*(S_R))))$  has finite codimension in $A^*(S_R)$.

This implies that $HK(R)$ is coseparable in $A^*(S_R)$ and by Theorem 
\ref{th.coseparable}, the map $R$ satisfies Sullivan's conjecture. 

By Corollary \ref{cor.coseparable}, the space $HK(R)$ coincides with $A^*(S_R)$ 
and by Lemma \ref{lm.invmeanempty} and Theorem \ref{th.finiteBanach}, the 
converse follows. 
\end{proof}

 Now we prove the following theorem.

\begin{theorem}\label{th.Evcompact} Let $R$ be a rational map without 
rotational domains. Then for any given critical value $v$ of $R$, the following 
statements are equivalent. 

\begin{enumerate}
\item The space $M(\hat{T}_v)$ is finite.
\item The restriction $E_v:B_0(S_R) \to \ell_\infty$ is weakly compact.
\item The restriction $E_v: HK(R)\rightarrow \ell_\infty$ is compact.
\item The operator $E_v:B(S_R)\rightarrow \ell_\infty$ is compact.
 
\end{enumerate}
\end{theorem}
\begin{proof}

Clearly, we have that (4) implies (3) and (3) implies (2). 

Let us show that (2) implies (4). Recall that $R$ does not admit non-trivial 
quasiconformal deformations on the Fatou set, that the measure of $P_R\cap SC(R)$ 
is zero, and $F_R$ does not have rotational domains. Hence by Theorem 
\ref{th.ThurstonMeanergodic}, the Thurston operator $T$ is mean-ergodic on 
$B_0(S_R)$, so $E_v(B_0)$ consists of convergent sequences. 
Since $\ell_1$ is isometrically isomorphic to $c^*$ then by duality, 
$E^*_v:\ell_1 \rightarrow A(S_R)=B^*_0(S_R)$ is given by 
$E^*_v(\{a_n\})=\sum_n a_n A_n(\gamma_v)$. If $E_v$ is weakly compact on 
$B_0(S_R)$ then $E^*_v$ is weakly compact on $\ell_1.$ The image of the 
canonical basis of $\ell_1$ is  $\{A_n(\gamma_v)\}$. By the mean ergodicity 
lemma, the sequence $\{A_n(\gamma_v)\}$ is precompact in norm. Then $E^*_v$  
is a compact operator. By duality the operator $E^{**}_v:B\rightarrow 
\ell_\infty$ given by $E_v|_{B_0(S_R)}^{**}(l)=E_v(l)=\{l(A_n(\gamma_v)\}$ is 
compact.

Now let us show that (1) implies (2). By Theorem \ref{th.finiteBanach}, 
$E_v(B(S_R))$ consists of convergent sequences. In other words,  $E_v$ defines 
a continuous operator from a Grothendieck space into a separable space. Since 
the unit ball in $\ell_1$ is sequentially precompact in the $*$-weak topology, 
by definition of a Grothendieck space, we have that $E_v$ is a weakly compact 
operator on $B(S_R).$ Hence the restriction of $E_v$ on $B_0(S_R)$ is also 
weakly compact. 

Finally, let us show that (2) implies (1). By Theorem \ref{th.finiteBanach} is 
enough to show that $E_v(B(S_R))$ consists of convergent sequences. But $E_v$ 
is 
a compact operator on $B(S_R)$ and $E_v(B(S_R))=E^{**}_v(B_0(S_R)\subset 
E_v(B_0(S_R))$. By mean ergodicity of $T$ on $B_0(S_R)$, the set 
$E_v(B_0(S_R))$ consists only of convergent sequences. 
\end{proof}

 The following corollary is one of the main results of this paper. 

\begin{corollary}\label{cor.Evcompact}
A rational map $R$ satisfies Sullivan's conjecture if and only if either the 
operator $Id-T$ is compact or the operator $E_v$ is compact for every critical 
value $v.$ 
\end{corollary}

\begin{proof}
First we proof the implication $\Rightarrow$. If  $R$ is either a Latt\`es map or 
the Julia set does not support a non-zero invariant Beltrami differential. If $R$ is Latt\`es 
the operator $Id-T$ is compact. The rest follows from Theorem \ref{th.Mtfinite} and 
Theorem \ref{th.Evcompact}. 

Next we proof the implication $\Leftarrow$. If $Id-T$ is compact then, by 
Theorem \ref{th.Qcompact}, the operator $T$ is uniformly ergodic on 
$B(S_R)$. If $R$ admits an invariant non-zero Beltrami differential 
supported on the Julia set then $1$ belongs to the spectrum of $T$. 
Then Theorem \ref{th.meanergosep} implies that the map 
$R$ is Latt\`es.
If $E_v$ is compact for every critical value $v$, then $R$ satisfies Sullivan's 
conjecture by Theorem \ref{th.Evcompact} and Theorem 
\ref{th.Mtfinite}. \end{proof}

Note that if $Id-T$ is compact then the operators $E_v$ are compact but 
the converse is no true. A map $R$ where $E_v\circ (Id-T)$ is not compact, 
for some critical value $v$, would serve as a counterexample to Sullivan's 
conjecture. However, we have the following observation which is one of the main 
motivations of the present work.

\begin{proposition}\label{prop.op.compact}
The operator $E_v \circ (Id-T):B(S_R) \to \ell_\infty$ is compact for every 
critical value $v$. 
\end{proposition}

\begin{proof}
A bounded sequence on $B(S_R)$ contains a subsequence which is $*$-weakly 
convergent. Let $\mu_i$ be a sequence on $B(S_R)$ with $*$-weak limit $\mu_0$. 
Define $\omega_i=E_v\circ (Id-T)(\mu_i)$ and $\omega_0=E_v\circ(Id-T)(\mu_0)$.
We will prove that $\omega_i$ converges to $\omega_0$ in norm. 

From the definition we have 
\begin{align*}
||\omega_i -\omega_0||&=\sup_n \left| \int_{S_R} 
A_n(\gamma_v)\cdot \left( Id-T\right)(\mu_i)-A_n(\gamma_v)\cdot \left(Id-T\right)(\mu_0) \right|
\\
&=\sup_n \left| \int_{S_R}  A_n(\gamma_v)\cdot \left(Id-T\right)(\mu_i-\mu_0)\right|.
\end{align*}
However, as $T$ is dual to the Ruelle operator we get  

\begin{align*}
\left | \int_{S_R}  A_n(\gamma_v)\cdot (Id-T)(\mu_i-\mu_0) \right|&=\left 
|\int_{S_R}(Id-R^*)\left( A_n(\gamma_v)\right)\cdot(\mu_i-\mu_0)\right|\\
&\leq \frac{2\|\gamma_v\|}{n}\| \mu_i-\mu_0\|.
\end{align*}

Since $\|\gamma_v\|\|\mu_i-\mu_0\|$ is bounded and $\mu_i$ converges $*$-weakly 
to $\mu_0$, then  $||\omega_i -\omega_0||$ converges to $0$ as $i\rightarrow 
\infty.$ Hence $E_v \circ (Id-T)$ is compact. \end{proof} 

In general, the compactness of the composition $E_v\circ (Id-T)$ does not imply 
the compactness of any of the factors. But this implication is true,  by 
Theorem \ref{th.Evcompact}, if and only if Sullivan's conjecture holds true. 

For every $v$, the operator $E_v$ has a canonical extension on
$L_\infty(\overline{\C})$ with the same defining formula. When 
$P_R\neq \overline{\C}$ and $SC(R)\cap P_R$ has Lebesgue measure zero,  the 
extension of $E_v$ on $L_\infty(\overline{\C})$ is compact if and only if $E_v$ is 
compact on $B(S_R).$ The extension of the operators  $E_v$ on $L_\infty(\overline{\C})$,
which we also denote by $E_v$, gives a sort of ``marking" for a rational map $R$. 
Furthermore, the operators $E_v$ induce a topology on the rational maps 
as follows.

A sequence of rational maps $R_i$ converges to $R_0$ in 
$v$-sense, where $v$ is a critical value of $R_0$ if and only if  for a given 
$\mu\in L_\infty(\C)$ there exists a sequence of critical values $v_i$ of $R_i$ 
such that $E_{v_i}(\mu)\rightarrow E_v(\mu)$ in $\ell_\infty.$

When $R_i$ converges to $R_0$ on the $v$-sense we will say that $R_0$ is a 
$v$-limit of $R_i.$

\begin{proposition}
If $R_0$ is a $v$-limit of $R_i$ such that $M(\hat{T}_{v_i})$ is finite  for 
all critical values $v_i$ of $R_i,$ then $M(\hat{T}_v)$ is also finite.
\end{proposition}

\begin{proof}
Let $c$ be the space of convergent sequences. By Theorem \ref{th.finiteBanach}, 
it is enough to show that the image $E_v$ belongs to $c$. 
By the hypothesis, for every critical value $w$ of $R_i$ the image of $E_{w}$ 
belongs to $c$. As the space  $c$ is closed in $\ell_\infty$ 
then the image of $E_v$ is also a subset of $c.$\end{proof}

In other words, roughly speaking, any $v$-limit of rational maps satisfying 
Sullivan's conjecture, satisfies Sullivan's conjecture too. Moreover, in 
general not every accumulation point in the $v$-sense has the same degree as 
the approximating elements. Further details on this topology will be the 
subject of a forthcoming work. 
 
\section{A mixing condition}

In this section we show that the Ruelle operator does not have fixed 
points when $R$ satisfies a kind of mixing condition over  
its strongly conservative set. 

We say that $R$ satisfies the $M$-\textit{condition} if, for an  invariant 
ergodic probability measure $\nu$ which is absolutely continuous with respect 
to the Lebesgue measure supported on the Julia set  $J_R$, $R$ satisfies the 
following two properties: 

\begin{enumerate}
\item  If $A$ and $B$ are $\nu$-measurable subsets of $SC(R)\cap J_R$, 
 
then $\lim \nu(B\cap R^{-n}(A))=\nu(A)\nu(B)$.

\item  There exists a $\nu$-measurable set $A_\nu\subset SC(R)\cap J_R$ with 
$\nu(A_\nu)>0$ such that the sequence of  functions 
$$
B^n(\chi_{SC(R)}(x))=\frac{\overline{(R^n)'(x)}}{(R^n)'(x)}
$$ 
is precompact in the topology of convergence in measure on $A_\nu$.
\end{enumerate}
The reader might recognize in the first property the classical mixing 
condition for invariant probability measures. We will comment on the second 
property at the end of this section.
If there is no invariant absolutely continuous probability measure, then the 
$M$-condition is vacuously satisfied. This is the case when the Lebesgue 
measure of the strongly conservative set is zero. 

\begin{proposition}\label{prop.converg} 
Assume $R$ satisfies the first property of the $M$-condition. 
Let $\nu$ be an invariant ergodic probability measure absolutely continuous 
with respect to Lebesgue. Let $W=supp(\nu)$ be the support of $\nu$. Then for 
every $\phi$ in $L_\infty(W,\nu)$, the sequence $$(|R^*|)^{*n}(\phi)$$ 
converges 
$*$-weakly to a constant. 
\end{proposition}
\begin{proof}
The proof follows from classical ergodic theory, for convenience to 
the reader we include it here. Let $\phi$ be a non negative function such that 
$d\nu=\phi(z) |dz|^2$. 
Since $\nu$ is an invariant probability measure we have  $|R^*|(\phi)=\phi$.  
Now consider the space $L_1(W,\nu)$ and the operator $S$  on $ L_1(W,\nu)$ 
given by $S(g)=\frac{1}{\phi}|R^{*}|(g\phi)$ 
with dual $S^*(\omega)=\omega(R)$ for $\omega$ in $L_1^*(W,\nu)$.  
Note that $S$ and $S^*$ are contractions in both $L_1(W,\nu)$ and 
$L_\infty(W,\nu)$. By well known ergodic theorems (see for instance Chapter 6 
of Dunford and Schwartz \cite{DunfordSchwartz}),  
both $S$ and $S^*$ are contracting mean-ergodic operators on $L_p(W,\nu)$ for 
all $1 \leq p <\infty$. The first part of the $M$-condition implies that, for 
every $f$ and $g$ in $L_2(W,\nu)$,
we have 
$$
\lim \int S^n(f)\,\overline{g}d\nu=\lim \int f\,\overline{S^{*n}(g)}d\nu=
\lim \int f \,\overline{g(R^n)}d\nu=\int f d\nu \int \overline{g} d\nu.
$$

Since $\nu$ is a probability measure, we get the chain of inclusions 
$$L_\infty(W,\nu)\subset L_2(W,\nu)\subset L_1(W,\nu),$$
and $L_2(W,\nu)$ defines an everywhere dense subspace in $L_1(W,\nu)$. 
Hence the orbits of $S$ and $S^*$ converge weakly 
in $L_1(W,\nu)$ and $*$-weakly in $L_\infty(W,\nu)$, respectively. 
Let $f_0$ be an element in $L_2(W,\nu)$. 
Then the  weak limits of 
$S^n(f_0)$ and $S^{n*}(f_0)$ are fixed points for $S$ and $S^*$, respectively. 
But $\nu$ is ergodic, 
so the spaces of fixed points of $S$ and $S^*$ consists only of constants. 
The conclusion of the proposition follows from the equality 
$S^*(\mu)=\mu(R)=|R^*|^*(\mu)$. 
\end{proof}

\begin{theorem}\label{th.mixing}
Assume that $R$ is not a Latt\`es map, that satisfies the $M$-condition, and 
$P_R\neq \C$. Then there is no non-zero fixed point of the 
Ruelle 
operator in $L_1(J_R)$.   
\end{theorem}
\begin{proof} Without loss of generality we can assume that $P_R$ is bounded. 
Now, assume that there exists a non-zero fixed point $f$ of the 
Ruelle operator in $L_1(J_R)$. 
Then by Proposition \ref{prop.invariantmeasure} there exists a fixed point the Beltrami operator $\mu$ with $\mu(z)=\frac{\overline{f(z)}}{|f(z)|}$ almost everywhere on $supp(f)$ and $|f|=\mu f$ is the density of a finite invariant measure $\nu$. After normalization we can assume that $\nu$ is a probability measure.  By Proposition \ref{prop.Krengel}, we have that $supp(f)\subset SC(R)$. Since $\mu(z)$ is not $0$ almost everywhere and $R$ is not a Latt\`es map then by the part (3) of  Lemma \ref{lemma.dichotomy} the support of 
$f$ is a bounded measurable subset of the postcritical set $P_R$.

For $S$ and $S^*$ as in the proof of Proposition \ref{prop.converg}, 
 consider the operator given by $$
Z(g)=\frac{1}{|f|}R^*(g|f|), 
$$ 
which defines an endomorphism of $L_1(supp(f),\nu)$. 
In this situation we have that  
$$
Z^*(\alpha)(z)=B(\alpha)(z)=\alpha(R)(z)\frac{\overline{R'}(z)}{R'(z)}
$$ 
defines an endomorphism of $L_\infty(supp(f),\nu)$. 
We obtain $Z(g)=\overline{\mu} S(\mu g)$ and $Z^*(\alpha)=\mu S^*(\overline{\mu}\alpha )$. 
By Proposition \ref{prop.converg} the orbits of $S$ and $S^*$ converge weakly  
to constants;  
hence the orbits of the operators $Z$ and $Z^*$ converge weakly  to scalar 
multiples of $\overline{\mu}$ and $\mu$ respectively. 
Let $c_g$ be the constant such that $Z^{*n}(g)$ converges weakly  to $ c_g\mu.$

Let $z_0$ be a density point of  $supp(f)$ and a continuity point of  
$\mu$. Since $supp(f)$ is a subset of the strongly conservative set $SC(R)$ and
almost every point $supp(f)$ is recurrent, we can assume that $z_0$ is 
also recurrent. 
This implies that there exists a sequence $\{n_i\}$ such that
$$
|\mu (R^{n_i}(z_0))-\mu(z_0)| \to 0.
$$ 
But $\mu$ is invariant, so that 
$\displaystyle{\left|\frac{\overline{(R^{n_i})'}}{(R^{n_i})'}(z_0)-1\right|}$ 
converges to $0$. 
Using together that $supp(f)$ is bounded and the $M$-condition holds,
we can assume that the previous sequence converges pointwise 
 almost everywhere in $A_\nu \subset supp(f)$. 
In this case we have 
$\displaystyle{\frac{\overline{(R^{n_i})'}}{R^{n_i}}(z_0)=(Z^*)^{n_i}(\chi_{
supp(f) } )(z_0).}$ 
By the Lebesgue dominated convergence theorem 
$\frac{\overline{(R^{n_i})'}}{R^{n_i}}$ converges to its pointwise almost 
everywhere limit 
 in the $L_1$ norm on $A_\nu$.  As 
norm and weak limits agree whenever they both exist we have  
$c_{\chi_{supp(f)}}=\frac{1}{\mu(z_0)}. $ But $c_{\chi_{supp(f)}}$ does not 
depend on the point $z_0$ nor on the sequence $\{n_i\}$.  
Therefore  $\mu(z)=\frac{1}{c_{\chi_{supp(f)}}}$ for almost every $z$ in 
$A_\nu$. Since $A_\nu\subset SC(R)$ there exists a natural $k_0$ such 
that $\nu(A_\nu \cap R^{k_0}(A_\nu))>0.$ Hence for a density point $y$ of $ 
A_\nu \cap R^{k_0}(A_\nu)$ there exists a density point $x\in A_\nu$ so that 
$y=R^{k_0}(x)$ and by invariance, we have 
$$c_{\chi_{supp(f)}}(x)=\mu(x)=\mu(y)\frac{\overline{(R^{k_0})'(x)}}{(R^{
k_0 })'(x)}=c_{\chi_{supp(f)}}(x)\frac{\overline{(R^{k_0})'(x)}}{(R^{
k_0 })'(x)}.$$ Then, again as in the proof of Proposition \ref{prop.chiA}, we 
have $A_\nu\cap (R^{k_0})^{-1}(A_\nu\cap R^{k_0}(A_\nu))\subset 
((R^{k_0})')^{-1}(\mathbb{R})$. 
But $\nu$ is absolutely continuous with respect to the Lebesgue measure and 
$\nu(A_\nu)>0$. 
This contradiction completes the proof.
\end{proof}

As an immediate corollary we have.

\begin{corollary}
If $R$ satisfies the conditions of Theorem \ref{th.mixing} and there exists a 
non-zero invariant Beltrami differential $\mu$, then $supp(\mu)\cap SC(R)$ has 
Lebesgue measure zero.
\end{corollary}

Finally, let us comment on the $M$-condition. According to M. Rees 
(see \cite{ReesErg}), the known examples of rational maps for which the 
strongly 
conservative set has positive Lebesgue measure forms a set of positive Lebesgue measure consisting of ergodic maps $R$ with $SC(R)=\overline{\C}.$ In other words there exists a
unique invariant absolutely continuous probability measure $\rho$ on $\overline{\C}$ 
so that for any pair $A,B$ of measurable subsets there exists  $n_0\in 
\mathbb{N}$ such that 
the Lebesgue measure of $R^{n_0}(A)\cap R^{n_0}(B)$ is positive. In ergodic 
theory this corresponds to the fact that the operator $S$ of Proposition 
\ref{prop.converg} has strongly convergent orbits in $L_1(\C,\rho)$. The first 
part of our definition of the $M$-condition is equivalent to weak convergence of 
orbits in 
$L_1(\C,\rho)$ with respect to $|R^*|$ and this is the classical definition of 
mixing dynamical systems with respect to a non-necessarily invariant measure. 

Now, about the second property of the $M$-condition: the precompactness in measure of a 
family of bounded measurable functions on probability measure spaces is a rather simple consequence of Koml\'os theorem (see for example page 39 of  \cite{AaronsonBook}) which states:

\textit{If $(X,\alpha)$ is a probability measure space and $f_n$ a sequence in $L_1(X,\alpha)$ with $\sup \|f_n\|_1<\infty.$ Then there exists $f_0\in L_1(X,\alpha)$ and a  subsequence $\{g_k := f_{n_k}\}\subset \{ f_n\} $ such that
for any subsequence $g_{k_i}$ of $g_k$ we have
$$\frac{1}{m}\sum_{i=1}^m g_{k_i}\rightarrow f_0,$$ $\alpha$-almost everywhere for $m\rightarrow \infty.$}

We call the sequence $g_k$, a \textit{Koml\'os subsequence} of $f_n.$ Then we have the following proposition.

\begin{proposition}
 Let $(X,\alpha)$ be a probability space and $g_k$ be a Koml\'os subsequence of a sequence of measurable functions $f_n$ with $|f_n|\leq M<\infty$
 almost everywhere. Then there exists a subset $A\subset X$ with $\alpha(A)>0$ such that $g_k$ 
 converges pointwise almost everywhere on $A.$
\end{proposition}
 
\begin{proof}
Let $f_0$ be a measurable function satisfying Koml\'os theorem for $g_k$, then $|f_0|\leq M$ 
almost everywhere. Let $x_0\in X$ be a point such that the functions  $f_0$ and $g_k$ are well-defined at $x_0$ and 
$\frac{1}{N}\sum_{k=1}^N g_k(x_0)\rightarrow f_0(x_0)$ as $N\rightarrow \infty.$

Let us show that the bounded sequence of complex numbers $\{g_k(x_0)\}$ converges to $f_0(x_0).$ 
Since $\{g_k(x_0)\}$ is bounded, it is sufficient to check that the accumulation set  of $g_k(x_0)$ consists 
of a single point $f_0(x_0).$ Assume that there exists a subsequence $g_{k_i}(x_0)$ converging
to $b\neq f_0(x_0)$ then the Ces\`aro averages $$\frac{1}{N}\sum_{i=1}^N g_{k_i}(x_0)\rightarrow b$$ 
as $N\rightarrow \infty$. But this contradicts that $g_k$ is a Koml\'os subsequence. \end{proof}

Since  pointwise convergence almost everywhere implies convergence in measure
in finite measure spaces, we conclude that  the second condition is always fulfilled on $SC(R)$.

Let us note that if $\mu$ is an 
invariant Beltrami differential then by Birkhoff's theorem we have 
$$
\frac{1}{n}\sum_{i=0}^{n-1}\overline{\mu}(R^i(z))=
\overline{\mu}(z)\left[\frac{1}{n}\sum_{
i=0}^{n-1}\left(\frac{\overline{R'^i(z)}}{R'^i(z)}\right)\right] \rightarrow 
\int \overline{\mu}(z) d\nu(z).
$$

Here the convergence is almost everywhere and  in  the $L_1$ norm on $supp(\nu)$, 
where $\nu$ is an invariant absolutely continuous probability measure. So the 
Ces\`aro averages $\frac{1}{n}\sum_{i=0}^{n-1} B^i(\chi_{SC(R)})$ 
converges almost everywhere to a multiple of $\mu$ on $supp(\mu)\cap 
supp(\nu)\subset SC(R).$  If every  Koml\'os subsequence of any subsequence of 
$B^n(\chi_{SC(R)})$ converges to a multiple of $\mu$ then the whole sequence $B^{n}(\chi_{SC(R)})$ converges to $\mu$ pointwise almost everywhere and hence converges in norm in  $L_1(\nu)$ by the Lebesgue dominated theorem. 

Therefore, the second part of the $M$-condition is an analog of weak mixing  of the complex Perron-Frobenius operator acting on $L_1(\nu)$ as the operator $Z$ defined on the proof of Theorem \ref{th.mixing}.

Finally, the next observation is the main motivation of this section. This 
proposition follows from classical ergodic theorems and a well known fact, due 
to Sullivan, which states that a measurable set $A\subset J_R$ has zero 
Lebesgue measure whenever the iterates $R^n$ are injective on $A$ and the 
Lebesgue measure of $R^n(A)\cap R^m(A)$ is $0$ for all distinct  $m,n>0$.

\begin{proposition}
Assume that a rational map $R$ is injective on $P_R\cap J_R$ 
and the Lebesgue measure of $P_R\cap SC(R)$ is $0$. Then the Lebesgue measure of $P_R$ is zero if and only if $R$ is mixing on its conservative part with respect 
to the Lebesgue measure restricted on $P_R$.
\end{proposition}

\begin{proof}
If the Lebesgue measure of $P_R$ is zero then clearly $R$ is mixing on $P_R$ 
with respect to Lebesgue measure.
On the contrary, assume that the Lebesgue measure of $P_R$ is positive. Let us 
consider the dynamics of $R$ restricted to $P_R$ and let $C(P_R)$ and $D(P_R)$ 
be the conservative and the dissipative parts of this action, respectively. As 
$R$ is injective on $P_R\cap J_R$, by Sullivan's lemma stated above, the Lebesgue 
measure of $D(P_R)$ is zero. Since the measure of $SC(R)\cap P_R$ is zero, then 
by Proposition \ref{prop.Krengel} there are no invariant absolutely continuous 
measures on $P_R$. 

Assume that $R$ is mixing on the conservative part of $P_R$. Using 
Theorem 1.4 on page 255 of \cite{Krengel} which states that every positive 
contraction $E$ on a $L_1$ space has orbits strongly convergent to $0$ whenever 
$E$ has weakly convergent orbits and $E$ has no non-zero fixed points. Therefore, for every $\phi\in L_1(C(R))$ the 
orbit  of $|R^*|^n(\phi)$ in  converges strongly to $0.$ As $R(C(R))=C(R)$ 
then $$ 0=\lim_{n\rightarrow \infty}\int_{C(R)}|R^{*n}|\chi_{C(R)}= 
\lim_{n\rightarrow 
\infty}\int_{R^{-n}(C(R))} \chi_{C(R)}=\int \chi_{C(R)}.$$

Thus the Lebesgue measure of $C(R)$ is $0$ and hence $P_R$ has also Lebesgue 
measure  $0$. \end{proof}

\bibliographystyle{amsplain}
\bibliography{workbib}

\providecommand{\bysame}{\leavevmode\hbox to3em{\hrulefill}\thinspace}
\providecommand{\MR}{\relax\ifhmode\unskip\space\fi MR }
\providecommand{\MRhref}[2]{%
  \href{http://www.ams.org/mathscinet-getitem?mr=#1}{#2}
}
\providecommand{\href}[2]{#2}
\begin{thebibliography}{10}

\bibitem{AaronsonBook}
J.~Aaronson, \emph{{An introduction to infinite ergodic theory}}, {Mathematical
  Surveys and Monographs}, vol.~50, American Mathematical Society, Providence,
  RI, 1997.

\bibitem{BonetWolf}
J.~Bonet and E.~Wolf, \emph{{A note on weighted {B}anach spaces of holomorphic
  functions}}, Arch. Math. (Basel) \textbf{81} (2003), no.~6, 650--654.
  \MR{2029241 (2004i:46037)}

\bibitem{CMdecomp}
C.~Cabrera and P.~Makienko, \emph{{On decomposable rational maps}}, Conform.
  Geom. Dyn. \textbf{15} (2011), 210--218. \MR{2869014 (2012k:37104)}

\bibitem{CabMakHyp}
C.~Cabrera and P.~Makienko, \emph{On hyperbolic metric and invariant {B}eltrami
  differentials for rational maps}, J. Geom. Anal. \textbf{28} (2018), no.~3,
  2346--2360.

\bibitem{DHTop}
A.~Douady and J.~H. Hubbard, \emph{{A proof of {T}hurston's topological
  characterization of rational functions}}, Acta Math. \textbf{171} (1993),
  263--297.

\bibitem{DunfordSchwartz}
N.~Dunford and J.T. Schwartz, \emph{{Linear operators. {P}art {I}}}, {Wiley
  Classics Library}, John Wiley \& Sons, Inc., New York, 1988, General theory,
  With the assistance of William G. Bade and Robert G. Bartle, Reprint of the
  1958 original, A Wiley-Interscience Publication. \MR{1009162 (90g:47001a)}

\bibitem{FonfLinRubinov}
V.~Fonf, M.~Lin, and A.~Rubinov, \emph{{On the uniform ergodic theorem in
  {B}anach spaces that do not contain duals}}, Studia Math. \textbf{121}
  (1996), no.~1, 67--85. \MR{1414895 (97i:47014)}

\bibitem{GamelinUniAlg}
T.~W. Gamelin, \emph{{Uniform algebras}}, Chelsea Publishing Co, NY., 1984.

\bibitem{GardLakic}
F.~Gardiner and N.~Lakic, \emph{{Quasiconformal {T}eichm{\"u}ller theory}},
  {Mathematical Surveys and Monographs}, vol.~76, American Mathematical
  Society, Providence, RI, 2000. \MR{1730906 (2001d:32016)}

\bibitem{Krabook}
I.~Kra, \emph{{Automorphic forms and {K}leinian groups}}, W. A. Benjamin, Inc.,
  Reading, Mass., 1972, Mathematics Lecture Note Series. \MR{0357775 (50
  \#10242)}

\bibitem{Krengel}
U.~Krengel, \emph{{Ergodic theorems}}, {de Gruyter Studies in Mathematics},
  vol.~6, Walter de Gruyter \& Co., Berlin, 1985, With a supplement by Antoine
  Brunel. \MR{797411 (87i:28001)}

\bibitem{Lotz}
H.~P. Lotz, \emph{{Uniform convergence of operators on {$L^\infty$} and similar
  spaces}}, Math. Z. \textbf{190} (1985), no.~2, 207--220. \MR{797538
  (87e:47032)}

\bibitem{L}
M.~Lyubich, \emph{{Dynamics of the rational transforms; the topological
  picture}}, Russian Math. Surveys (1986).

\bibitem{LyuTypical}
M.~Yu. Lyubich, \emph{{Typical behavior of trajectories of the rational mapping
  of a sphere}}, Dokl. Akad. Nauk SSSR \textbf{268} (1983), no.~1, 29--32.
  \MR{687919 (84f:30036)}

\bibitem{MSS}
R.~Ma{{\~n}{\'e}}, P.~Sad, and D.~Sullivan, \emph{{On the dynamics of rational
  maps}}, Ann. Scien. Ec. Norm. Sup. Paris(4) (1983).

\bibitem{MakRuelle}
P.~Makienko, \emph{{Remarks on the Ruelle operator and the invariant line
  fields problem: II}}, Ergodic Theory and Dynamical Systems \textbf{25}
  (2005), no.~05, 1561--1581.

\bibitem{Mc1}
C.~McMullen, \emph{{Complex dynamics and renormalization}}, {Annals of
  Mathematics Studies}, vol. 135, Princeton University Press, Princeton, NJ,
  1994.

\bibitem{McMSull}
C.~McMullen and D.~Sullivan, \emph{{Quasiconformal homeomorphisms and dynamics.
  {III}. {T}he {T}eichm{\"u}ller space of a holomorphic dynamical system}},
  Adv. Math. \textbf{135} (1998), no.~2, 351--395.

\bibitem{MilLat}
J.~Milnor, \emph{{On {L}att{\`e}s maps}}, {Dynamics on the {R}iemann sphere},
  Eur. Math. Soc., Z{\"u}rich, 2006, pp.~9--43. \MR{2348953 (2009h:37090)}

\bibitem{HarmandWWerner}
{P. Harmand}, D.~Werner, and W.~Werner, \emph{{{$M$}-ideals in {B}anach spaces
  and {B}anach algebras}}, {Lecture Notes in Mathematics}, vol. 1547,
  Springer-Verlag, Berlin, 1993. \MR{1238713}

\bibitem{ReesErg}
M.~Rees, \emph{{Positive measure sets of ergodic rational maps}}, Ann. Sci.
  {\'E}cole Norm. Sup. (4) \textbf{19} (1986), no.~3, 383--407.

\bibitem{RuelleZeta}
D.~Ruelle, \emph{{Zeta-functions for expanding maps and {A}nosov flows}},
  Invent. Math. \textbf{34} (1976), no.~3, 231--242.

\bibitem{Taomeasure}
T.~Tao, \emph{{An introduction to measure theory}}, {Graduate Studies in
  Mathematics}, vol. 126, American Mathematical Society, Providence, RI, 2011.

\bibitem{ZdunikParab}
A.~Zdunik, \emph{{Parabolic orbifolds and the dimension of the maximal measure
  for rational maps}}, Invent. Math. \textbf{99} (1990), no.~3, 627--649.
  \MR{1032883 (90m:58120)}

\end{thebibliography}

\end{document}